\documentclass[preprint,11pt]{elsarticle}

\usepackage{times,amsmath}
\usepackage{amsthm}
\usepackage{amsmath,algorithm2e,bm}
\usepackage{amsmath,amssymb,amsfonts,amscd}
\usepackage{graphicx,latexsym,color}
\usepackage{hyperref}
\usepackage{multirow}
\usepackage{floatrow}
\usepackage{wrapfig}
\usepackage{verbatim}
\usepackage{epsfig}
\usepackage{lineno,hyperref}
\usepackage{mathrsfs}
\usepackage{xcolor}
\usepackage{natbib}
\usepackage{enumitem}
\biboptions{sort&compress}

 \usepackage{lineno}
\numberwithin{equation}{section}
\newproof{pf}{Proof}

\newtheorem{thm}{Theorem}[section]
\newtheorem{lem}[thm]{Lemma}

\newtheorem{remark}[thm]{Remark}
\def \bes{\begin{eqnarray}}
\def \ees{\end{eqnarray}}
\def \bns{\begin{eqnarray*}}
\def \ens{\end{eqnarray*}}

\newcommand{\dif}{\,\mathrm{d}}
\newenvironment{eqa}{\begin{equation}%
  \begin{array}{rcl}}{\end{array}\end{equation}}
\newcommand\beqa{\begin{eqa}}
\newcommand\eeqa{\end{eqa}}
\newcommand{\re}[1]{\mbox{$($\ref{#1}$)$}}\baselineskip 1pt

\newcommand{\p}{\partial}

\journal{Journal of Differential Equations }

\renewcommand{\tnotemark}[1]{\textsuperscript{#1}}
\usepackage{subfig}

\begin{document}

\begin{frontmatter}


\author[label3]{Xinyue Evelyn Zhao}
 \ead{xzhao45@utk.edu}
  \affiliation[label1]{organization={Department of Mathematics, University of Tennessee},
	city={Knoxville},
	postcode={TN 37916},
	country={USA}}

\author[label2]{Junping Shi\corref{cor1}}
\ead{jxshix@wn.edu}
  \cortext[cor1]{Corresponding author}
  \affiliation[label2]{organization={Department of Mathematics, William $\&$ Mary},
 	city={Williamsburg},
 	postcode={VA 23187-8795},
 	country={USA}}

\title{On determination of the bifurcation type for a free boundary problem modeling tumor growth}

\begin{abstract}
    Many mathematical models in different disciplines involve the formulation of free boundary problems, where the domain boundaries are not predefined. These models present unique challenges, notably the nonlinear coupling between the solution and the boundary, which complicates the identification of bifurcation types. This paper mainly investigates the structure of symmetry-breaking bifurcations in a two-dimensional free boundary problem modeling tumor growth. By expanding the solution to a high order with respect to a small parameter and computing the bifurcation direction at each bifurcation point, we demonstrate that all the symmetry-breaking bifurcations occurred in the model, as established by the Crandall-Rabinowitz Bifurcation From Simple Eigenvalue Theorem, are pitchfork bifurcations. These findings reveal distinct behaviors between the two-dimensional and three-dimensional cases of the same model.
\end{abstract}


\end{frontmatter}


\section{Introduction}
In recent decades, an increasing number of PDE models describing solid tumor growth in the form of free boundary problems have been proposed and studied, see \cite{bazaliy2003global,bazaliy2003free,cui2009well,cui2007bifurcation,fontelos2003symmetry,friedman2007mathematical,friedman2006bifurcation,friedman2006asymptotic,friedman2008stability,friedman1999analysis,friedman2001symmetry} and reference therein. These models, which consider the tumor tissue as a density of proliferating cells, are based on reaction-diffusion equations and mass conservation law for cell density and nutrient concentration within the tumor. The influences of different factors on tumor growth are investigated, such as the effect of angiogenesis \cite{FFH1,angio2,angio1,huang2021asymptotic}, time delay \cite{he2022linear,he2022thelinear,zhao2020impact,zhao2020symmetry}, inhibitor \cite{cui2002analysis,wang2014bifurcation}, cell cycle \cite{friedman2007mathematical, friedman2008multiscale, friedman2009free, friedman2010cell}, necrotic core \cite{cui2001analysis,hao2012bifurcation,lu2023bifurcation}, and so on. A systematic survey of tumor model studies was provided in \cite{lowengrub2}.

Let $\sigma$ and $p$ denote the concentration of nutrients and the pressure, respectively. The basic tumor growth model is to find the unknown tumor region $\Omega(t)$ (or the tumor boundary $\p \Omega(t)$) and the unknown functions $\sigma(x,t)$ and $p(x,t)$ satisfying
\begin{align}
    \label{model1}\sigma_t - \Delta \sigma + \sigma = 0 \hspace{2em}&\text{in }\Omega(t),\; t>0,\\
    \label{model2}\Delta p = -\mu(\sigma - \widetilde{\sigma}) \hspace{2em}&\text{in }\Omega(t),\; t>0,\\
    \label{model3}\sigma=1 \hspace{2em} &\text{on }\partial \Omega(t),\; t>0,\\
    \label{model4}p=\kappa \hspace{2em} &\text{on }\partial \Omega(t),
\end{align}
where $\kappa$ is the mean curvature, and
\begin{equation}
    \label{model5}V_n = -\frac{\p p}{\p n} \hspace{2em}\text{on }\partial \Omega(t),
\end{equation}
where $\frac{\p}{\p n}$ is the derivative along the outward normal $\vec{n}$ and $V_n$ is the velocity of the free boundary $\p \Omega(t)$ in the outward normal direction $\vec{n}$. The initial conditions are
\begin{equation}
    \label{model6} \sigma(x,0) = \sigma_0(x) \hspace{2em}\text{in }\Omega(0)\text{, \,where $\Omega(0)$ is given.}
\end{equation}

In the basic model \re{model1} -- \re{model6}, it is assumed that the tumor region $\Omega(t)$ contains just one type of cell, and the cell density is uniform. The tumor will either expand or shrink depending on the amount of nutrients within the tumor, and the tumor proliferation rate is assumed to be linear with respect to the concentration of nutrients, given by the function $\mu(\sigma-\widetilde{\sigma})$. Here, $\mu$ is a parameter expressing the ``intensity" of tumor expansion due to mitosis (if $\sigma > \widetilde{\sigma})$ or tumor shrinkage by apoptosis (if $\sigma < \widetilde{\sigma})$, and $\widetilde{\sigma}$ is a threshold concentration. In addition, it is assumed that the tumor region is a porous medium, so that Darcy's law $\vec{V}=-\nabla p$ holds. Combining it with the law of conservation of mass $\text{div} \vec{V} = \mu(\sigma-\widetilde{\sigma})$, we derive the equation \re{model2}. Furthermore, the boundary condition for the pressure $p$, i.e., the equation \re{model4}, is due to cell-to-cell adhesiveness, and the continuity of the velocity field $\vec{V}\cdot\vec{n} = V_n$ yields the relation \re{model5}.

It is well-established that under the assumption $0<\widetilde{\sigma}<1$, the system \re{model1} -- \re{model6} admits a unique radially symmetric stationary solution for both the 2-D \cite{friedman2001symmetry} and 3-D \cite{friedman1999analysis} cases. It was also proved in \cite{friedman2001symmetry} for the 2-D case and \cite{fontelos2003symmetry,friedman2008stability} for the 3-D case that there exists a sequence of symmetry-breaking bifurcation branches consisting non-symmetric stationary solutions that bifurcate from the branch of the unique radially symmetric stationary solution with parameter $\mu$. Similar bifurcation results in various tumor growth models have been obtained in \cite{cui2007bifurcation,hao2012bifurcation,lu2023bifurcation,wang2014bifurcation,zhao2020symmetry}. Yet, the structure of these bifurcations remains largely unexplored. It was demonstrated in \cite{friedman2008stability} that the first symmetry-breaking bifurcation of model \re{model1} -- \re{model6} in 3-D case is a transcritical bifurcation. In contrast, our findings in this paper reveal a distinct behavior in the 2-D version of the same model. \textit{We will prove that all symmetry-breaking bifurcations for model \re{model1} -- \re{model6} in 2-D are pitchfork bifurcations.} The methods in this paper can be applied to analyze tumor growth models with different factors proposed in \cite{cui2002analysis,cui2001analysis,FFH1,angio2,hao2012bifurcation,he2022linear,he2022thelinear,angio1,huang2021asymptotic,lu2023bifurcation,wang2014bifurcation,zhao2020impact,zhao2020symmetry}. Although the computations can be more complicated, similar bifurcation results could be obtained. This finding is consistent with the numerical simulations in \cite{hao2012bifurcation,hao2012continuation}.

The structure of this paper is as follows: Section 2 presents preliminary results. In Section 3, we prove the existence of bifurcation of non-radially symmetric solutions. Section 4 is dedicated to deriving the type of bifurcations and proving that all bifurcation points are pitchfork bifurcations. Finally, in Section 5, we provide a discussion of our results.

\section{Preliminaries}

\subsection{Radially symmetric stationary solution.} Consider \re{model1} -- \re{model6} in the two-dimensional spatial domain. We denote the radially symmetric stationary solution of \re{model1} -- \re{model6} by $(\sigma_S(r),p_S(r),\p B_{R_S})$, where $B_{R_S}$ denotes the disk centered at 0 with radius $R_S$. From \re{model1} -- \re{model6}, the solution satisfies
\begin{eqnarray}
    \label{rss1}-\Delta \sigma_S + \sigma_S = 0 \hspace{2em}&0<r<R_S,\\
    \label{rss2}\Delta p_S = -\mu(\sigma_S-\widetilde{\sigma}) \hspace{2em}&0<r<R_S,\\
    \label{rss3}\sigma_S = 1\hspace{2em} &r=R_S,\\
    \label{rss4}p_S = \frac{1}{R_S} \hspace{2em} & r=R_S,\\
    \label{rss5}V_n = -\frac{\p p_S}{\p r} = 0 \hspace{2em} &r=R_S.
\end{eqnarray}
The system has now been reduced to an ODE system, and the explicit solution is given by
\begin{align}
    \label{sigmaS}\sigma_S(r) &= \frac{I_0(r)}{I_0(R_S)},\\
    \label{pS} p_S(r) &= \frac14\mu\widetilde{\sigma} r^2 - \mu\frac{I_0(r)}{I_0(R_S)} + \frac{1}{R_S} + \mu - \frac{1}{4}\mu\widetilde{\sigma}R_S^2,
\end{align}
where $R_S$ is uniquely determined by the equation
\begin{equation}
    \label{RS}\frac{\widetilde{\sigma}}{2} = \frac{I_1(R_S)}{R_S I_0(R_S)}.
\end{equation}
In \re{sigmaS} -- \re{RS}, the functions $I_n(r)$, where $n$ are non-negative integers, represent the modified Bessel functions of the first kind. For convenience, we collect some properties of these functions in Subsection 2.2.

{\color{black} Later on, we will use $\mu$ as the bifurcation parameter, which explicitly appears in the formula of $p_S$. To avoid the dependence of this particular solution $(\sigma_S,p_S,R_S)$ on the bifurcation parameter $\mu$, we decompose $p_S$ as
$$p_S(r) = \widetilde{p}_S(r) + \mu p_S^*(r),$$
where $\widetilde{p}_S$ and $p_S^*$ are defined to satisfy the following boundary value problems, respectively:
\begin{equation}\label{eqn:pst}
    \left\{
\begin{aligned}
    &-\Delta \widetilde{p}_S = 0 \hspace{2em} & 0<r<R_S,\\
    &\widetilde{p_S} = \frac{1}{R_S} \hspace{2em} & r=R_S,
\end{aligned}\right.
\end{equation}
\begin{equation}\label{eqn:pss}
    \left\{
\begin{aligned}
    &-\Delta p_S^* =\sigma_S-\widetilde{\sigma} \hspace{2em} & 0<r<R_S,\\
    &p_S^* = 0 \hspace{2em} & r=R_S.
\end{aligned}\right.
\end{equation}
These equations can be solved to obtain:
\begin{eqnarray}
    &&\widetilde{p}_S(r) = \frac{1}{R_S},\label{pst}\\
    &&p_S^*(r) =\frac14 \widetilde{\sigma} r^2 - \frac{I_0(r)}{I_0(R_S)} + 1 - \frac14 \widetilde{\sigma} R_S^2.\label{pss}
\end{eqnarray}
}

In later computations, we will use the information of $\frac{\p \sigma_S(R_S)}{\p r}$, $\frac{\p^2 \sigma_S(R_S)}{\p r^2}$, $\frac{\p \widetilde{p}_S(R_S)}{\p r}$, $\frac{\p p^*_S(R_S)}{\p r}$, $\frac{\p^2 \widetilde{p}_S(R_S)}{\p r^2}$, $\frac{\p^2 p^*_S(R_S)}{\p r^2}$, $\frac{\p^3 \widetilde{p}_S(R_S)}{\p r^3}$, and $\frac{\p^3 p^*_S(R_S)}{\p r^3}$. We put them in the following lemma:
\begin{lem}\label{lem:deri}
For the radially symmetric stationary solution $(\sigma_S,\widetilde{p}_S,p^*_S,\p B_{R_S})$, we have
    \begin{align}
        \label{ds1}&\frac{\p \sigma_S(R_S)}{\p r} = \frac{I_1(R_S)}{I_0(R_S)},\\
        \label{ds2}&\frac{\p^2 \sigma_S(R_S)}{\p r^2} = 1-\frac{I_1(R_S)}{R_S I_0(R_S)},\\
        \label{dpt} &\frac{\p \widetilde{p}_S(R_S)}{\p r} = \frac{\p^2 \widetilde{p}_S(R_S)}{\p r^2} =  \frac{\p^3 \widetilde{p}_S(R_S)}{\p r^3} =0,
    \end{align}
and, by \re{pS}, we also have
\begin{align}
    \label{dps1}&\frac{\p p^*_S(R_S)}{\p r} = \frac{\widetilde{\sigma}}{2}R_S - \frac{I_1(R_S)}{I_0(R_S)}=0,\\
        \label{dps2}&\frac{\p^2 p^*_S(R_S)}{\p r^2} = \frac{\widetilde{\sigma}}{2} -\Big(1-\frac{I_1(R_S)}{R_SI_0(R_S)}\Big) = \frac{2I_1(R_S)}{R_S I_0(R_S)}-1,\\
        \label{dps3}&\frac{\p^3 p^*_S(R_S)}{\p r^3}= \frac{1}{R_S} - \frac{(2+R_S^2)I_1(R_S)}{R_S^2 I_0(R_S)}.
\end{align}
\end{lem}

\subsection{Properties of Bessel Functions}
Recall that the modified Bessel function of the first kind $I_n(\xi)$ 
satisfies the differential equation
\begin{equation}\label{bessel1}
    I''_n(\xi) + \frac{1}{\xi}I'_n(\xi) - \bigg(1+\frac{n^2}{\xi^2}\bigg) I_n(\xi) = 0\hspace{2em}\xi>0,
\end{equation}
and is given by
\begin{equation}
    \label{bessel2}
    I_n(\xi) = \bigg(\frac{\xi}{2}\bigg)^n \sum_{k=0}^\infty \frac{1}{k!\Gamma(n+k+1)}\bigg(\frac{\xi}{2}\bigg)^{2k},
\end{equation}
from which it is easy to derive
\begin{equation}
    \label{besseln}
\frac{I_{n+1}(\xi)}{I_n(\xi)} < \frac{\xi}{2n} , \hspace{2em} \mbox{for } \xi>0, \; n\ge 1.
\end{equation}
Furthermore, by \cite{fontelos2003symmetry, friedman2006bifurcation, friedman2001symmetry}, $I_n(\xi)$ satisfies the following properties, which are needed in subsequent computations.
\begin{gather}
I'_n(\xi) + \frac{n}{\xi}I_n(\xi) = I_{n-1}(\xi) \hspace{2em} n\ge 1, \label{bessel3}\\
I'_n(\xi) - \frac{n}{\xi}I_n(\xi) = I_{n+1}(\xi) \hspace{2em} n\ge 0, \label{bessel4}\\
\xi^{n+1} I_n(\xi) = \frac{\p}{\p \xi}(\xi^{n+1} I_{n+1}(\xi)) \hspace{2em} n\ge 0,\label{bessel5}\\
I_{n-1}(\xi) - I_{n+1}(\xi) = \frac{2n}{\xi}I_n(\xi) \hspace{2em} n\ge 1, \label{bessel6}\\
\frac{I_n(\xi)}{\xi} \text{ is increasing in $\xi$ for } \xi>0 \hspace{2em} n\ge 1.\label{bessel7}
\end{gather}

\subsection{Bifurcation theory} In this subsection, we state some abstract bifurcation theorems which are critical in our analysis:
\begin{thm}\label{CRthm}  {\bf (Crandall-Rabinowitz Theorem, \cite{crandall1971bifurcation,Liu2007,nirenberg1974topics})}
Let $X$, $Y$ be real Banach spaces and let $\mathcal{F}(\cdot,\cdot)$ be a $C^p$ map, $p\ge 2$, of a neighborhood $(\mu_0,0)$ in $\mathbb{R}\times X$ into $Y$. Denote by $\mathcal{F}_x$ and $\mathcal{F}_{\mu x}$ the first- and second-order Fr\'echet derivatives, respectively.
Assume the following four conditions hold:
\begin{itemize}
\item[(I)] $\mathcal{F}(\mu,0) = 0$ for all $\mu$ in a neighborhood of $\mu_0$,
\item[(II)] $\mathrm{Ker} \,\mathcal{F}_x(\mu_0,0)$ is one dimensional space, spanned by $x_0$,
\item[(III)] $\mathrm{Im} \, \mathcal{F}_x(\mu_0,0)=Y_1$ has codimension one,
\item[(IV)] $\mathcal{F}_{\mu x} (\mu_0,0) [x_0] \notin Y_1$,
\end{itemize}
then $(\mu,x)=(\mu_0,0)$ is a bifurcation point of the equation $\mathcal{F}(\mu,x)=0$ in the following sense: in a neighborhood of $(\mu,x)=(\mu_0,0)$, the set of solutions to $\mathcal{F}(\mu,x) =0$ consists of two $C^{p-1}$ smooth curves, $\Gamma_1$ and $\Gamma_2$, which intersect only at the point $(\mu_0,x)=(\mu_0,0)$; $\Gamma_1$ is the curve $x\equiv 0$, and $\Gamma_2$ can be parameterized as follows:
$$\Gamma_2=(\mu(\varepsilon),x(\varepsilon)): |\varepsilon| \text{ small, } (\mu(0),x(0))=(\mu_0,0),\; x'(0)=x_0.$$
\end{thm}
Theorem \ref{CRthm} was first proved in \cite{crandall1971bifurcation} for $C^1$ maps with continuous $\mathcal{F}_{\mu x}$. The version here with $C^p$ ($p\ge 2$) maps was stated in \cite{nirenberg1974topics} with $C^{p-2}$ solution curves. Indeed, the solution curves are $C^{p-1}$ as pointed out in \cite{Liu2007} (see Corollary 2.3 and Lemma 2.5).

While the Crandall-Rabinowitz Theorem reveals the locations where bifurcation branches occur, it does not provide information on the direction of bifurcation branches, which is important in determining the type of bifurcation diagrams. To investigate the type of bifurcation, we need the following theorem:

\begin{thm}{(\cite{Liu2007,shi1999persistence})}\label{shi}
  Suppose the conditions of Theorem \ref{CRthm} are satisfied and let $(\mu_0,0)$ be the bifurcation point of $\mathcal{F}(\mu,x)=0$ in Theorem \ref{CRthm}. Along the bifurcation branch $\Gamma_2$, we have
    \begin{equation}
        \label{mu-d}
        \mu'(0) = -\frac{\langle l,\mathcal{F}_{xx}(\mu_0,0)[x_0,x_0]\rangle}{2\langle l,\mathcal{F}_{\mu x}(\mu_0,0)[x_0]\rangle},
    \end{equation}
    where $l\in Y^*$ satisfying $\text{Ker}\, l = \text{Im}\, \mathcal{F}_x(\mu_0,0)$, and $\langle \cdot, \cdot\rangle$ is the the duality pair of $Y^*$ and $Y$. If $\mu'(0)\neq 0$, which indicates $\mathcal{F}_{xx}(\mu_0,0)[x_0,x_0] \notin \text{Im}\,\mathcal{F}_x(\mu_0,0)$, then the bifurcation branch $\Gamma_2$ exhibit a transcritical bifurcation. On the other hand, if $\mu'(0)=0$, then $\mathcal{F}_{xx}(\mu_0,0)$ $[x_0,x_0] \in \text{Im}\,\mathcal{F}_x(\mu_0,0)$, and the bifurcation is a pitchfork type. Furthermore, in the case of a pitchfork bifurcation and assuming $p\ge 3$, the direction of bifurcation at $(\mu_0,0)$ is determined  by
    \begin{equation}\label{eqn:mu-d2}
        \mu''(0) = -\frac{\langle l,\mathcal{F}_{xxx}(\mu_0,0)[x_0,x_0,x_0]\rangle + 3\langle l,\mathcal{F}_{xx}(\mu_0,0)[x_0,\phi]\rangle}{3\langle l,\mathcal{F}_{\mu x}(\mu_0,0)[x_0]\rangle},
    \end{equation}
    where $\phi$ is the solution of
    \begin{equation}\label{eqn:phi}
        \mathcal{F}_{xx}(\mu_0,0)[x_0,x_0] + \mathcal{F}_{x}(\mu_0,0)[\phi] = 0.
    \end{equation}
\end{thm}

\section{The existence of bifurcation branches}
In this section, we give an alternative proof to \cite{friedman2001symmetry} regarding the existence of symmetry-breaking bifurcation branches in the stationary problem for system \re{model1} -- \re{model6}. This proof employs the Crandall-Rabinowitz Theorem and the methodologies are similar to those used in \cite{friedman2008stability,angio1,wang2014bifurcation,zhao2021convergence,zhao2020symmetry}.

{\color{black}In the context of the Crandall-Rabinowitz Theorem, we define the curve $\Gamma_1$ as the branch of radially symmetric stationary solution discussed in Section 2.1. Specifically, $\Gamma_1$ is given by
$$\Gamma_1 = \{(\mu,\sigma_S,\widetilde{p}_S,p_S^*,\partial B_{R_S}):\,\mu>0\}.$$
Based on this radially symmetric solution branch,
we consider a family of domains $\Omega_\varepsilon$ with perturbed boundaries
$$\partial \Omega_\varepsilon: \, r=R_S + \widetilde{R}(\theta),$$
where $\widetilde{R}(\theta)=\varepsilon S(\theta)$ with $|\varepsilon|\ll 1$. Within the perturbed domain $\Omega_\varepsilon$, we denote $(\sigma,p)$ by the unique solution of the system
\begin{align}
    \label{ps1}-\Delta \sigma + \sigma = 0 \hspace{2em} &\text{in }\Omega_\varepsilon,\\
    \label{ps2}-\Delta p = \mu(\sigma - \widetilde{\sigma}) \hspace{2em}&\text{in }\Omega_\varepsilon,\\
    \label{ps3}\sigma = 1\hspace{2em} &\text{on } \partial \Omega_\varepsilon,\\
    \label{ps4}p = \kappa \hspace{2em} &\text{on } \partial \Omega_\varepsilon.
\end{align}
Similar to the decomposition in Section 2.1, we decompose $p$ as
\begin{equation}\label{decom}
    p=\widetilde{p} + \mu p^*,
\end{equation}
where $\widetilde{p}$ and $p^*$ satisfy the following boundary value problems:
\begin{equation}\label{eqn:pt}
    \left\{
    \begin{aligned}
    &-\Delta \widetilde{p} = 0 \hspace{2em}&\text{in }\Omega_\varepsilon,\\
    &\widetilde{p} = \kappa \hspace{2em}&\text{on }\p \Omega_\varepsilon,
\end{aligned}
    \right.
\end{equation}
and
\begin{equation}\label{eqn:ps}
    \left\{
    \begin{aligned}
    &-\Delta p^* = \sigma-\widetilde{\sigma} \hspace{2em}&\text{in }\Omega_\varepsilon,\\
    &p^* = 0 \hspace{2em}&\text{on }\p \Omega_\varepsilon,
\end{aligned}
    \right.
\end{equation}
respectively. This decomposition allows us to eliminate the dependence of the solution on the bifurcation parameter $\mu$. It is clear that the system \re{ps1} -- \re{ps4} is equivalent to solving for $(\sigma,\widetilde{p},p^*)$ from \re{ps1}, \re{ps3}, \re{eqn:pt}, and \re{eqn:ps}. For notation simplicity, we shall refer to this new system as System $(\mathcal{A})$.
}

We  define the bifurcation equation $\mathcal{F}$ as
\begin{equation}
\begin{split}
    \label{F}
    \mathcal{F}(\mu,\widetilde{R}) = \frac{\p p}{\p n}\bigg|_{\p \Omega_\varepsilon}.
\end{split}
\end{equation}
Combining with \re{decom}, we have
\begin{equation}
\begin{split}
    \label{F1}
    \mathcal{F}(\mu,\widetilde{R}) = \Big(\frac{\p \widetilde{p}}{\p n} + \mu\frac{\p p^*}{\p n}\Big)\bigg|_{\p \Omega_\varepsilon}
\end{split}
\end{equation}
By \re{model5}, $\mathcal{F}(\mu,\widetilde{R})$ represents the negative value of the normal velocity of the free boundary. In a stationary solution, the free boundary remains unchanged. Therefore, $(\sigma,\widetilde{p},p^*)$ is a stationary solution of System $(\mathcal{A})$ in the perturbed domain $\Omega_\varepsilon$ if and only if $\mathcal{F}(\mu,\widetilde{R}) = 0$.

The function $S(\theta)$ may be viewed as a function defined on the unit circle
$$\Sigma = \{x\in \mathbb{R}^2:\;|x|=1\},$$
so it is natural to impose $2\pi$-periodic boundary condition on the function $S(\theta)$. Furthermore, it can be proved that the solution to System $(\mathcal{A})$ is even in variable $\theta$ if we assume $S(\theta)=S(-\theta)$. As a result, we introduce the following Banach spaces: for any integer $l\ge0$ and $0<\alpha<1$,
\begin{gather}
    X^{l+\alpha} = \{\widetilde{R}\in C^{l+\alpha}(\Sigma): \widetilde{R} \text{ is $2\pi$-periodic in $\theta$}, \text{and } \widetilde{R} \text{ is even}\},\label{bsp1}\\
    X^{l+\alpha}_2 = \text{closure of the linear subspace spanned by } \{\cos(j\theta),\;j=0,2,4,\cdots\} \text{ in } X^{l+\alpha}.\label{bsp2}
\end{gather}

For System $(\mathcal{A})$, one can apply Schauder theory to establish the following lemma:
\begin{lem}\label{lemnew3.1}
    If $S\in C^{3+\alpha}(\Sigma)$ and $(\sigma,\widetilde{p},p^*)$ satisfies System $(\mathcal{A})$, then $\sigma\in C^{3+\alpha}(\overline{\Omega}_\varepsilon)$, $\widetilde{p}\in C^{1+\alpha}(\overline{\Omega}_\varepsilon)$, and $p^*\in C^{3+\alpha}(\overline{\Omega}_\varepsilon)$.
\end{lem}

Lemma \ref{lemnew3.1} shows that the mapping $(\mu,\widetilde{R})\mapsto \mathcal{F}(\mu,\widetilde{R})$ is from $\mathbb{R}^+ \times C^{3+\alpha}(\Sigma)$ to $C^{\alpha}(\Sigma)$. Recalling the definitions of Banach spaces $X^{l+\alpha}$ and $X_2^{l+\alpha}$ in \re{bsp1} and \re{bsp2}, we can use similar arguments to prove $\mathcal{F}(\mu,\widetilde{R})$ maps from $X^{l+3+\alpha}$ to $X^{l+\alpha}$ (or from $X^{l+3+\alpha}_2$ to $X^{l+\alpha}_2$) for any integer $l\ge 0$.

{\color{black}In order to apply the Crandall-Rabinowitz Theorem, we need to compute the Fr\'echet derivatives of $\mathcal{F}$. To do that, we shall analyze expansions of $(\sigma,\widetilde{p},p^*)$ in $\varepsilon$. For any $\mu>0$, we formally write
\begin{gather}
    \label{expands}\sigma = \sigma_S + \varepsilon\sigma_1 +  O(\varepsilon^2),\\
    \label{expandpt}\widetilde{p} = \widetilde{p}_S + \varepsilon \widetilde{p}_1 +  O(\varepsilon^2),\\
    \label{expandps}p^* = p^*_S + \varepsilon p_1^* +  O(\varepsilon^2).
\end{gather}
Substituting equations \re{expands} -- \re{expandps} into  System $(\mathcal{A})$, neglecting terms of order $O(\varepsilon^2)$, and also recalling that
\begin{equation}\label{kappa}
    \kappa = \frac{(R_S+\varepsilon S)^2 + 2S_\theta^2 - (R_S+\varepsilon S) \cdot S_{\theta\theta}}{\big((R_S+\varepsilon S)^2 + S_\theta^2\big)^{3/2}}  =\frac{1}{R_S}-\varepsilon\frac{1}{R_S^2}\Big(S+S_{\theta\theta}\Big) + O(\varepsilon^2),
\end{equation}
we obtain the linearized systems for $\sigma_1$, $\widetilde{p}_1$, and $p^*_1$:
\begin{equation}\label{eqn:sigma1}
    \left\{
    \begin{aligned}
        &-\Delta \sigma_1 + \sigma_1 = 0 \hspace{2em} &\text{in } B_{R_S},\\
        &\sigma_1(R_S,\theta) = -\frac{\p \sigma_S(R_S)}{\p r} S(\theta) \hspace{2em}&\text{on } \p B_{R_S},
    \end{aligned}
    \right.
\end{equation}
\begin{equation}
    \label{eqn:p1t}
    \left\{
    \begin{aligned}
     &-\Delta \widetilde{p}_1 = 0 \hspace{2em}&\text{in }B_{R_S},\\
    & \widetilde{p}_1(R_S,\theta) = -\frac{1}{R_S^2}(S(\theta)+S_{\theta\theta}(\theta)) -\frac{\p \widetilde{p}_S(R_S)}{\p r}S(\theta) \hspace{2em}&\text{on }\p B_{R_S},
    \end{aligned}
    \right.
\end{equation}
\begin{equation}
    \label{eqn:p1s}
    \left\{
    \begin{aligned}
     &-\Delta p^*_1 = \sigma_1 \hspace{2em}&\text{in }B_{R_S},\\
    & p^*_1(R_S,\theta) = -\frac{\p p_S^*(R_S)}{\p r}S(\theta) \hspace{2em}&\text{on }\p B_{R_S}.
    \end{aligned}
    \right.
\end{equation}
In the subsequent Section 4, we will consider the next order of approximation and prove a more refined formula for $\kappa$.

Following \cite{friedman2008stability,angio1,wang2014bifurcation,zhao2020symmetry}, the following lemma can be easily proved.

\begin{lem}\label{lemnew3.2}
If $S\in C^{3+\alpha}(\Sigma)$ and $(\sigma,\widetilde{p},p^*)$ satisfies System $(\mathcal{A})$, then
\begin{eqnarray}
    &\|\sigma-\sigma_S\|_{C^{3+\alpha}(\overline{\Omega}_\varepsilon)} \le C|\varepsilon| \|S\|_{C^{3+\alpha}(\Sigma)},\\
    &\|(\widetilde{p}+\mu p^*)-(\widetilde{p}_S+\mu p^*_S)\|_{C^{1+\alpha}(\overline{\Omega}_\varepsilon)} \le C|\varepsilon| \|S\|_{C^{3+\alpha}(\Sigma)},
\end{eqnarray}
where the constant $C$ is independent of $\varepsilon$.
\end{lem}

To further justify equations \re{expands} -- \re{expandps}, we need to estimate the $O(\varepsilon^2)$ terms using the appropriate norms. It is important to note that $(\sigma,\widetilde{p},p^*)$ is defined only on $\Omega_\varepsilon$, $(\sigma_S,\widetilde{p}_S,p^*_S)$ has explicit expressions and is defined everywhere on $\mathbb{R}^2$, and $(\sigma_1,\widetilde{p}_1,p^*_1)$ is defined only on $B_{R_S}$. Therefore, we must transform these functions into a common domain before proceeding with our analysis. To this end, we introduce the Hanzawa transformation $H_\varepsilon$, which is defined by
\begin{equation}
    \label{Hanzawa}
    (r,\theta) = H_\varepsilon(r',\theta')\equiv (r'+\chi(R_S -r')\varepsilon S(\theta'),\theta'),
\end{equation}
where $\chi(z):{\mathbb R}\to {\mathbb R}$ is a bounded function satisfying
\begin{equation*}
    \chi \in C^\infty,\hspace{2em} \chi(z) = \left\{
    \begin{aligned}
        &0, \hspace{1em} \text{if } |z|\ge 3\delta_0/4\\
        &1, \hspace{1em} \text{if } |z|< \delta_0/4
    \end{aligned}
    \right. , \hspace{2em} \bigg|\frac{\p^k \chi}{\p z^k}\bigg| \le \frac{C}{\delta_0^k}, \;\; k=1,2,\cdots,
\end{equation*}
and $\delta_0$ is a small positive constant. The Hanzawa transformation $H_\varepsilon$ maps  $B_{R_S}$ into $\Omega_\varepsilon$ while keeping the ball $\{r < R_S - \frac34 \delta_0\}$ fixed to avoid the singularity of the Laplace operator at $0$. The inverse transformation $H^{-1}_\varepsilon$ maps $\Omega_\varepsilon$ onto $B_{R_S}$.

Using the Hanzawa transformation, we let
\begin{equation}
    \label{Htrans}
    \widehat{\sigma}_1(r,\theta) = \sigma_1(H_\varepsilon^{-1}(r,\theta)),\hspace{2em} \widehat{\widetilde{p}}_1(r,\theta) = \widetilde{p}_1(H_\varepsilon^{-1}(r,\theta)), \hspace{2em} \widehat{p}^*_1(r,\theta)= p^*_1(H_\varepsilon^{-1}(r,\theta)).
\end{equation}
Then, $(\sigma,\widetilde{p},p^*)$, $(\sigma_S,\widetilde{p}_S,p^*_S)$, and $(\widehat{\sigma}_1,\widehat{\widetilde{p}}_1,\widehat{p}^*_1)$ are all defined on the same domain $\Omega_\varepsilon$. This allows us to establish another lemma. The proof of the lemma follows methods similar to those found in \cite{friedman2008stability, angio1, wang2014bifurcation, zhao2020symmetry}. We introduce additional decompositions for
$p$, $p_S$, and $p_1$ here; however, these do not impact the validity of the proof.

\begin{lem}
    If $S\in C^{3+\alpha}(\Sigma)$,  $(\sigma,\widetilde{p},p^*)$ satisfies System $(\mathcal{A})$ and $(\widehat{\sigma}_1,\widehat{\widetilde{p}}_1,\widehat{p}^*_1)$ is defined in \re{Htrans} where $(\sigma_1,\widetilde{p}_1,p^*_1)$ is the solution to the system \re{eqn:sigma1} -- \re{eqn:p1s}, then
    \begin{eqnarray}
    &\|\sigma-\sigma_S-\varepsilon \widehat{\sigma}_1 \|_{C^{3+\alpha}(\overline{\Omega}_\varepsilon)} \le C|\varepsilon|^2 \|S\|_{C^{3+\alpha}(\Sigma)},\label{ests}\\
    &\|(\widetilde{p}+\mu p^*)-(\widetilde{p}_S+\mu p^*_S) - \varepsilon(\widehat{\widetilde{p}}_1 + \mu \widehat{p}^*_1)\|_{C^{1+\alpha}(\overline{\Omega}_\varepsilon)} \le C|\varepsilon| \|S\|_{C^{3+\alpha}(\Sigma)},\label{estp}
\end{eqnarray}
where the constant $C$ is independent of $\varepsilon$.
\end{lem}

Since
$$\mathcal{F}(\mu,0) = \Big(\frac{\p \widetilde{p}_S}{\p r} + \mu\frac{\p p^*_S}{\p r}\Big)\bigg|_{r=R_S} = 0+ \mu\bigg(\frac{\widetilde{\sigma}}2 R_S - \frac{I_1(R_S)}{I_0(R_S)}\bigg)=0,$$
where this last equality is justified by \re{RS}, we proceed to substitute equations \re{expandpt}, \re{expandps}, and \re{estp} into \re{F1} to derive the following:
\begin{equation*}
    \begin{split}
        \mathcal{F}(\mu,\widetilde{R}) &= \Big(\frac{\p \widetilde{p}}{\p n}+ \mu\frac{\p p^*}{\p n}\Big)\bigg|_{r=R_S+\varepsilon S}  = \Big(\nabla \widetilde{p} \cdot \vec{n}+ \mu\nabla p^* \cdot \vec{n}\Big)\bigg|_{r=R_S+\varepsilon S} \\
        &= \Big( \frac{\p \widetilde{p}}{\p r} \vec{e}_r +  \frac{1}{r}\frac{\p \widetilde{p}}{\p \theta} \vec{e}_\theta\Big)\cdot \frac{1}{\sqrt{1+\big(\frac{r_\theta}{r}\big)^2}}\Big(\vec{e}_r - \frac{r_\theta}{r}\vec{e}_\theta\Big)\bigg|_{r=R_S+\varepsilon S}\\
        &\hspace{2em}+ \mu\Big( \frac{\p p^*}{\p r} \vec{e}_r +  \frac{1}{r}\frac{\p p^*}{\p \theta} \vec{e}_\theta\Big)\cdot \frac{1}{\sqrt{1+\big(\frac{r_\theta}{r}\big)^2}}\Big(\vec{e}_r - \frac{r_\theta}{r}\vec{e}_\theta\Big)\bigg|_{r=R_S+\varepsilon S}\\
        &= \frac{\p \widetilde{p}}{\p r}\bigg|_{r=R_S+\varepsilon S} + \mu \frac{\p p^*}{\p r}\bigg|_{r=R_S+\varepsilon S} + O(|\varepsilon|^2\|S\|_{C^{3+\alpha}(\Sigma)}) \\
        &= \Big(\frac{\p \widetilde{p}_S(R_S)}{\p r} + \mu\frac{\p p^*_S(R_S)}{\p r}\Big) + \Big(\frac{\p^2 \widetilde{p}_S(R_S)}{\p r^2}\varepsilon S + \mu \frac{\p^2 p^*_S(R_S)}{\p r^2}\varepsilon S + \frac{\p \widetilde{p}_1(R_S,\theta)}{\p r}\varepsilon \\
        &\hspace{2em} \mu \frac{\p p^*_1(R_S,\theta)}{\p r}\varepsilon \Big) + O(|\varepsilon|^2\|S\|_{C^{3+\alpha}(\Sigma)}).
    \end{split}
\end{equation*}
Thus,
$$\mathcal{F}(\mu,\widetilde{R}) = \mathcal{F}(\mu,0) +\varepsilon\Big(\frac{\p^2 \widetilde{p}_S(R_S)}{\p r^2}S + \mu \frac{\p^2 p^*_S(R_S)}{\p r^2}S  + \frac{\p \widetilde{p}_1(R_S,\theta)}{\p r} + \mu\frac{\p p^*_1(R_S,\theta)}{\p r}\Big) + O(|\varepsilon|^2\|S\|_{C^{3+\alpha}(\Sigma)})$$
which formally gives the Fr\'echet derivative of $\mathcal{F}$ in $\widetilde{R}$ at $(\mu,0)$
\begin{equation}
    \label{FreD}
    \mathcal{F}_{\widetilde{R}}(\mu,0) [S]= \frac{\p^2 \widetilde{p}_S(R_S)}{\p r^2}S + \mu \frac{\p^2 p^*_S(R_S)}{\p r^2}S  + \frac{\p \widetilde{p}_1(R_S,\theta)}{\p r} + \mu\frac{\p p^*_1(R_S,\theta)}{\p r}.
\end{equation}
In what follows, we shall use \re{FreD} to establish the existence of bifurcation branches by verifying the regularity and the four assumptions in Theorem \ref{CRthm}. The results are stated in the following Theorem:
\begin{thm}\label{main1}
   For any $\widetilde{\sigma}>0$, there exists a unique $R_S>0$ such that System $(\mathcal{A})$ has a family of radially symmetric stationary solutions in a form of
   \begin{equation*}\label{Gamma1}
      \Gamma_1=\{(\mu, \sigma_S(\cdot),\widetilde{p}_{S}(\cdot),p^*_S(\cdot),\partial B_{R_S}): \mu>0\}.
   \end{equation*}
   Then, for every even integer $n\ge2$, the point $(\mu_n,0)$ is a bifurcation point of System $(\mathcal{A})$, where
   \begin{equation}
       \label{thm:mun}
       \mu_n = \frac{n(n^2-1)}{R_S^3 M_n}
   \end{equation}
   and $M_n$ is defined in \re{Mn}.
   Furthermore,  in a neighborhood of $(\mu_n,0)$, the set of solutions to System $(\mathcal{A})$ consists of exactly $\Gamma_1$ and
   \begin{equation*}\label{Gamma2}
\Gamma_{2,n}=\{(\mu_n(\varepsilon),\sigma_n(\varepsilon,\cdot),\widetilde{p}_n(\varepsilon,\cdot),p^*_n(\varepsilon,\cdot),\partial \Omega_{\varepsilon,n}): |\varepsilon|\ll 1\},
   \end{equation*}
such that $\sigma_n(\varepsilon,\cdot)=\sigma_S(\cdot)+\varepsilon\sigma_{1}(\cdot)+O(\varepsilon^2)$, $\widetilde{p}_n(\varepsilon,\cdot)=\widetilde{p}_S(\cdot)+\varepsilon \widetilde{p}_{1}(\cdot)+O(\varepsilon^2)$, $p^*_n(\varepsilon,\cdot)=p^*_S(\cdot)+\varepsilon p^*_{1}(\cdot)+O(\varepsilon^2)$, where $\sigma_{1}$, $\widetilde{p}_{1}$, and $p^*_{1}$ satisfy \eqref{eqn:sigma1} -- \re{eqn:p1s} with $S(\theta)=\cos(n \theta)$, respectively, and
the corresponding free boundary $\partial\Omega_{\varepsilon,n}$ is of the form $r=R_S+\varepsilon \cos(n\theta) + O(\varepsilon^2)$.
\end{thm}}

\begin{proof}
    To begin with, for any $\mu>0$, the radially symmetric stationary solution $(\sigma_S,\widetilde{p}_S,p^*_S)$ exists when $\widetilde{R}(\theta)=0$ and remains independent of $\mu$. Consequently, the first assumption of Theorem \ref{CRthm}, namely that $\mathcal{F}(\mu,0)=0$, is satisfied.

By \re{estp}, the operator $\mathcal{F}$ maps from $\mathbb{R}^+\times X^{l+3+\alpha}$ to $X^{l+\alpha}$ (or from $\mathbb{R}^+\times X^{l+3+\alpha}_2$ to $X^{l+\alpha}_2$). Since the set $\{\cos(n\theta)\}_{n=0}^\infty$ is an orthonormal basis for $X^{l+\alpha}$, we use a Fourier series expression for $S(\theta)$:
\begin{equation}
    \label{S}
    S(\theta)=\sum_{n=0}^\infty a_n \cos(n\theta).
\end{equation}
Applying the separation of variables technique to the systems  \re{eqn:sigma1} -- \re{eqn:p1s} and utilizing the results from Lemma \ref{lem:deri}, we can solve $(\sigma_1,\widetilde{p}_1,p^*_1)$ explicitly by
\begin{gather}
    \label{sigma1}\sigma_1(r,\theta)= \sum_{n=0}^\infty a_n \sigma_{1,n}(r)\cos(n\theta) = \sum_{n=0}^\infty a_n \Big(-\frac{I_1(R_S)I_n(r)}{I_0(R_S)I_n(R_S)}\Big)\cos(n\theta),\\
    \label{p1t}\widetilde{p}_1(r,\theta)=\sum_{n=0}^\infty a_n \widetilde{p}_{1,n}(r)\cos(n\theta)=\sum_{n=0}^\infty a_n \frac{(n^2-1)r^n}{R_S^{n+2}} \cos(n\theta),\\
    \label{p1s}p^*_1(r,\theta)=\sum_{n=0}^\infty a_n p^*_{1,n}(r)\cos(n\theta) = \sum_{n=0}^\infty a_n \Big(-\frac{r^n I_1(R_S)}{R_S^n I_0(R_S)} + \frac{I_1(R_S)I_n(r)}{I_0(R_S)I_n(R_S)}\Big)\cos(n\theta).
\end{gather}
Differentiating \re{sigma1} -- \re{p1s} in $r$, and using properties of Bessel functions in Section 2.2, we obtain
\begin{equation*}
    \label{sigma1d} \frac{\p \sigma_1(r,\theta)}{\p r} = \sum_{n=0}^\infty a_n \frac{\p \sigma_{1,n}(r)}{\p r}\cos(n\theta) = \sum_{n=0}^\infty a_n\Big[ -\frac{I_1(R_S)}{I_0(R_S)I_n(R_S)}\Big(I_{n+1}(r)+\frac{n}{r}I_n(r)\Big)\Big] \cos(n\theta),
\end{equation*}
\begin{equation*}
    \label{p1dt} \frac{\p  \widetilde{p}_1(r,\theta)}{\p r} = \sum_{n=0}^\infty a_n \frac{\p \widetilde{p}_{1,n}(r)}{\p r}\cos(n\theta)= \sum_{n=0}^\infty a_n \frac{n(n^2-1) r^{n-1}}{R_S^{n+2}}\cos(n\theta),
\end{equation*}
and
\begin{equation*}
    \begin{aligned}\frac{\p p^*_1(r,\theta)}{\p r} &= \sum_{n=0}^\infty a_n \frac{\p p^*_{1,n}(r)}{\p r} \cos(n\theta)\\
    &= \sum_{n=0}^\infty a_n \Big[-\frac{nr^{n-1}I_1(R_S)}{R_S^n I_0(R_S)} + \frac{I_1(R_S)}{I_0(R_S)I_n(R_S)}\Big(I_{n+1}(r)+\frac{n}{r}I_n(r)\Big) \Big]\cos(n\theta).
    \end{aligned}
\end{equation*}
Hence, in \re{FreD},
\begin{equation}
    \label{p1d1}
    \begin{split}
    \frac{\p \widetilde{p}_1(R_S,\theta)}{\p r} + \mu\frac{\p p^*_1(R_S,\theta)}{\p r}  =\sum_{n=0}^\infty a_n \Big[\frac{n(n^2-1)}{R_S^3} + \mu\frac{I_1(R_S)I_{n+1}(R_S)}{I_0(R_S)I_n(R_S)}\Big] \cos(n\theta).
    \end{split}
\end{equation}
Substituting \re{dpt}, \re{dps2}, and \re{p1d1} into \re{FreD}, we have
\begin{equation}\label{Fre-d}
\begin{split}
    &\mathcal{F}_{\widetilde{R}}(\mu,0)[S(\theta)] \\
    =&\mu\Big(\frac{2I_1(R_S)}{R_SI_0(R_S)}-1\Big)\sum_{n=0}^\infty a_n\cos(n\theta) + \sum_{n=0}^\infty a_n \Big[\frac{n(n^2-1)}{R_S^3} + \mu\frac{I_1(R_S)I_{n+1}(R_S)}{I_0(R_S)I_n(R_S)}\Big] \cos(n\theta)\\
    =&\sum_{n=0}^\infty a_n \Big(\mu \frac{2I_1(R_S)}{R_SI_0(R_S)}-\mu + \frac{n(n^2-1)}{R_S^3} + \mu\frac{I_1(R_S)I_{n+1}(R_S)}{I_0(R_S)I_n(R_S)}\Big)\cos(n\theta),
\end{split}
\end{equation}
in particular,
\begin{equation}\label{Fre-d1}
\begin{split}
    \mathcal{F}_{\widetilde{R}}(\mu,0) [\cos(n\theta)] &= \Big(\mu \frac{2I_1(R_S)}{R_SI_0(R_S)}-\mu + \frac{n(n^2-1)}{R_S^3} + \mu\frac{I_1(R_S)I_{n+1}(R_S)}{I_0(R_S)I_n(R_S)}\Big)\cos(n\theta)\\
    &\triangleq \Big(-\mu M_n + \frac{n(n^2-1)}{R_S^3}\Big) \cos(n\theta),
\end{split}
\end{equation}
where
\begin{equation}
    \label{Mn}
    M_n = 1 - \frac{2I_1(R_S)}{R_S I_0(R_S)} - \frac{I_1(R_S)I_{n+1}(R_S)}{I_0(R_S)I_n(R_S)}.
\end{equation}
It was proved in \cite{zhao2020impact} (Lemma 4.1) and \cite{huang2019linear} (Lemma 3.3) that
\begin{equation}\label{ineq}
    1-\frac{2I_1(x)}{xI_0(x)} - \frac{I_1(x)I_{n+1}(x)}{I_0(x)I_n(x)}> 0\hspace{2em}\text{for } n\ge2\text{ and } x>0,
\end{equation}
thus
\begin{equation}
    \label{Mn-sign}
    M_n > 0  \hspace{2em} \text{for $n\ge 2$}.
\end{equation}
Furthermore, for $n=0$ (see Lemma 4.1 in \cite{zhao2020symmetry}),
$$\mathcal{F}_{\widetilde{R}}(\mu,0) [1 ]= \mu\Big(\frac{2I_1(R_S)}{R_SI_0(R_S)} -1 - \frac{I_1^2(R_S)}{I_0^2(R_S)}\Big) > 0;$$
and for $n=1$,
$$\mathcal{F}_{\widetilde{R}}(\mu,0) [\cos(\theta)] = \mu\Big(\frac{2I_1(R_S)}{R_SI_0(R_S)}  -1 +  \frac{I_2(R_S)}{I_0(R_S)}\Big) \cos(\theta)= 0,$$
implying that $\cos(\theta)$ is in the kernel of $\mathcal{F}_{\widetilde{R}}(\mu,0)$ for all $\mu$.
For $n\ge 2$, following \re{Mn-sign},  $\mathcal{F}_{\widetilde{R}}(\mu,0) [\cos(n\theta)]=0$ if and only if
\begin{equation*}
    \label{mu_n}\mu=\mu_n\triangleq \frac{n(n^2-1)}{R_S^3 M_n} .
\end{equation*}
It was proved in \cite{huang2021asymptotic,zhao2020symmetry} that $\mu_n$ is monotonically increasing in $n$ for $n\ge 2$. Therefore,
\begin{equation*}
    \text{Ker} \mathcal{F}_{\widetilde{R}}(\mu,0) = \left\{\begin{aligned}
        &\text{Span}\{\cos(\theta), \cos(n\theta)\} \hspace{2em}&\text{ if }\mu = \mu_n, \; n\ge 2;\\
        &\text{Span}\{\cos(\theta)\} \hspace{2em}&\text{if }\mu\neq \mu_n, \; n \ge 2.
    \end{aligned}\right.
\end{equation*}
To ensure that $\text{Ker} \mathcal{F}_{\widetilde{R}}(\mu,0)$ has a dimension of $1$, thereby meeting requirement (II) of Theorem \ref{CRthm}, we must exclude the case $n=1$. Consequently, we work with the space $X_2^{3+\alpha}$ defined in \re{bsp2}. For any even integer $n\ge2$,
\begin{gather*}
    \label{Ker}\text{Ker}\,\mathcal{F}_{\widetilde{R}}(\mu_n,0) = \text{Span}\{\cos(n\theta)\},\\
    \label{Im}\text{Im}\,\mathcal{F}_{\widetilde{R}}(\mu_n,0)=\text{Span}\{1,\cos(2\theta),\cos(4\theta),\cdots,\cos((n-2)\theta),\cos((n+2)\theta),\cdots\},
\end{gather*}
which meet the requirements (II) and (III) of Theorem \ref{CRthm}. Finally, by differentiating \re{Fre-d1} in $\mu$, we obtain
\begin{equation}\label{F-mu-r}
    \mathcal{F}_{\mu\widetilde{R}}(\mu_n,0) [1,\cos(n\theta)] = -M_n \cos(n\theta) \notin \text{Im }\mathcal{F}_{\widetilde{R}}(\mu_0,0),
\end{equation}
which fulfills the last requirement of Theorem \ref{CRthm}.

In summary, all the requirements of the Crandall-Rabinowitz Theorem (Theorem \ref{CRthm}) are satisfied at $(\mu_n,0)$ for even integers $n\ge 2$. Therefore, for each such $n$, $(\mu_n,0)$ is a bifurcation point of System $(\mathcal{A})$, and the conclusions in Theorem \ref{main1} hold.
\end{proof}

\begin{remark}\label{main2}
The bifurcation result is actually valid for all integers $n\ge2$ not restricting to even $n$ only. For any odd $n\ge 3$, we may work with the Banach space
$$M^{l+\alpha}=\text{closure of the linear space spanned by } \{\cos(j\theta),\;j=0,2,3,4,5,\cdots\} \text{ in } X^{l+\alpha}$$
and apply the Crandall-Rabinowitz theorem in a more delicate manner. The issue here is that $\mathcal{F}$ does not map $M^{l+3+\alpha}$ into $M^{l+\alpha}$, but we could shift the center of the system to eliminate the mode $n=1$ and make a modified mapping $\widetilde{\mathcal{F}}$ which maps $M^{l+3+\alpha}$ into $M^{l+\alpha}$. In \cite{pan2022symmetry}, the authors utilized group-theoretic methods by selecting appropriate isotropy subgroups to demonstrate that the bifurcation results are also valid for odd integers $n\ge2$. Although the results are derived in a 3-D context, similar methods could be applicable to the 2-D case. However, the details of the analysis, which involve complex computations, will not be discussed here.
\end{remark}

\section{The type of bifurcation points}
By Theorem \ref{main1} and Remark \ref{main2}, we know that, for every integer $n\ge 2$, the bifurcating solution $(\mu_n(\varepsilon),\widetilde{R}_n(\varepsilon))$ satisfies
\begin{gather*}
    \mathcal{F}(\mu_n(\varepsilon),\widetilde{R}_n(\varepsilon))=0,\\
    (\mu_n(0),\widetilde{R}_n(0)) = (\mu_n,0),\\
    \widetilde{R}_n(\varepsilon) =\varepsilon \cos(n\theta) + O(\varepsilon^2).
\end{gather*}
In order to determine the type of the bifurcation, we need to compute $\mu_n'(0)$ by using \re{mu-d}. To do that, we need $\varepsilon^2$-order expansion, and we formally write (without the subscript $n$):
\begin{gather}
    \label{hexpand1} \sigma(r,\theta) = \sigma_S(r) + \varepsilon\sigma_1(r,\theta) + \varepsilon^2 \sigma_2(r,\theta) + O(\varepsilon^3),\\
    \label{hexpand2} \widetilde{p}(r,\theta) = \widetilde{p}_S(r) +\varepsilon \widetilde{p}_1(r,\theta) +\varepsilon^2 \widetilde{p}_2(r,\theta) + O(\varepsilon^3),\\
    \label{hexpand3} p^*(r,\theta) = \widetilde{p}_S(r) +\varepsilon p^*_1(r,\theta) +\varepsilon^2 p^*_2(r,\theta) + O(\varepsilon^3),
\end{gather}

We also establish a refined formula for the mean curvature $\kappa$ on of the free boundary.

\begin{lem}\label{kappalem}
    If $\p\Omega_\varepsilon:\, r=R_S+\varepsilon S(\theta)$, where $S\in C^2(\Sigma)$, then
    \begin{equation}\label{kappa2}
    \kappa\Big|_{r=R_S+\varepsilon S(\theta)} = \frac{1}{R_S} -\varepsilon\frac{S+S_{\theta\theta}}{R_S^2} + \varepsilon^2 \frac{2S S_{\theta\theta} + S^2 + \frac12 S_\theta^2}{R_S^3}
    + O(\varepsilon^3).
    \end{equation}
\end{lem}
\begin{proof}
    Using the mean curvature formula in the two-dimensional case for a curve $r=\rho(\theta)$:
    \begin{equation}\label{kappa-c1}
    \begin{split}
        \kappa = \frac{\rho^2 + 2\rho_\theta^2 - \rho \cdot \rho_{\theta\theta}}{(\rho^2 + \rho_\theta^2)^{3/2}} = \frac{1+2(\frac{\rho_\theta}{\rho})^2-\frac{\rho_{\theta\theta}}{\rho}}{\rho(1+(\frac{\rho_\theta}{\rho})^2)^{3/2}}=\frac{1}{\rho} \bigg(1+2\Big(\frac{\rho_\theta}{\rho}\Big)^2-\frac{\rho_{\theta\theta}}{\rho}\bigg)\bigg(1+\Big(\frac{\rho_\theta}{\rho}\Big)^2\bigg)^{-\frac32}.
    \end{split}
    \end{equation}
Taking $\rho(\theta)=R_S+\varepsilon S(\theta)$, we have $\rho_\theta =\varepsilon S_\theta$ and $\rho_{\theta\theta} = \varepsilon S_{\theta\theta}$. Then
\begin{equation*}
    \begin{split}
        \frac{1}{\rho} &= \frac{1}{R_S + \varepsilon S} = \frac{1}{R_S} - \varepsilon\frac{S}{R_S^2} + \varepsilon^2 \frac{S^2}{R_S^3} - \varepsilon^3 \frac{S^3}{R_S^4} + O(\varepsilon^4),\\
        \frac{\rho_\theta}{\rho}  &= \varepsilon \frac{S_\theta}{R_S} -\varepsilon^2 \frac{S S_\theta}{R_S^2} + \varepsilon^3\frac{S^2 S_\theta}{R_S^3} + O(\varepsilon^4),\\
        \Big(\frac{\rho_\theta}{\rho}\Big)^2  &= \varepsilon^2 \frac{S_\theta^2}{R_S^2} - \varepsilon^3\frac{2S S_\theta^2}{R_S^3} +O(\varepsilon^4),\\
        \frac{\rho_{\theta\theta}}{\rho}  &= \varepsilon \frac{S_{\theta\theta}}{R_S} -\varepsilon^2 \frac{S S_{\theta\theta}}{R_S^2} + \varepsilon^3\frac{S^2 S_{\theta\theta}}{R_S^3} + O(\varepsilon^4),
    \end{split}
\end{equation*}
and
\begin{equation*}
\begin{split}
        \bigg(1+2\Big(\frac{\rho_\theta}{\rho}\Big)^2-\frac{\rho_{\theta\theta}}{\rho}\bigg) &= 1 - \varepsilon \frac{S_{\theta\theta}}{R_S}+ \varepsilon^2 \frac{2S_\theta^2 + S S_{\theta\theta}}{R_S^2} - \varepsilon^3\frac{4S S_\theta^2+S^2 S_{\theta\theta}}{R_S^3} + O(\varepsilon^4),\\
        \bigg(1+\Big(\frac{\rho_\theta}{\rho}\Big)^2\bigg)^{-\frac32} &= 1 -\frac{3}{2}\Big(\frac{\rho_\theta}{\rho}\Big)^2 + O(\varepsilon^4) = 1-\varepsilon^2 \frac32\frac{S_\theta^2}{R_S^2} + \varepsilon^3 \frac{3 S S_\theta^2}{R_S^3} + O(\varepsilon^4).
        \end{split}
\end{equation*}
Putting these into \re{kappa-c1}, we obtain
\begin{equation*}
    \begin{split}
        \kappa =& \bigg(\frac{1}{R_S} - \varepsilon\frac{S}{R_S^2} + \varepsilon^2 \frac{S^2}{R_S^3} - \varepsilon^3 \frac{S^3}{R_S^4} + O(\varepsilon^4)\bigg) \bigg(1 - \varepsilon \frac{S_{\theta\theta}}{R_S}+ \varepsilon^2 \frac{2S_\theta^2 + S S_{\theta\theta}}{R_S^2} \\
        &\;- \varepsilon^3\frac{4S S_\theta^2+S^2 S_{\theta\theta}}{R_S^3}+O(\varepsilon^4) \bigg)\bigg(1-\varepsilon^2 \frac32\frac{S_\theta^2}{R_S^2} + \varepsilon^3 \frac{3 S S_\theta^2}{R_S^3} + O(\varepsilon^4)\bigg)\\
        =&\frac{1}{R_S} -\varepsilon\frac{S+S_{\theta\theta}}{R_S^2} + \varepsilon^2 \frac{2S S_{\theta\theta} + S^2 + \frac12 S_\theta^2}{R_S^3} - \varepsilon^3 \frac{S^3+\frac32 S S_\theta^2 + 3S^2 S_{\theta\theta} - \frac32 S_\theta^2 S_{\theta\theta}}{R_S^4} + O(\varepsilon^4),
    \end{split}
\end{equation*}
which gives the estimate \re{kappa2}.
\end{proof}

Given our focus on the type of bifurcation at a bifurcation point $\mu=\mu_n$, we will henceforth consider the special case where $S(\theta)=\cos(n\theta)$ and $\mu=\mu_n$ for $n\ge 2$, where $\mu_n$ is given in Theorem \ref{main1}. Each $\cos(n\theta)$ corresponds to a bifurcation branch emanating from $(\mu_n,0)$.

Substituting the refined expansion \re{hexpand1} into \re{ps1} and \re{ps3} and collecting the $\varepsilon^2$-order terms, we derive the system for $\sigma_2$:
\begin{equation}\label{eqn:sigma2}
    \left\{
    \begin{aligned}
        &-\Delta \sigma_2 + \sigma_2 = 0 \hspace{2em} &\text{in } B_{R_S},\\
        &\sigma_2(R_S,\theta) = -\frac12 \frac{\p^2 \sigma_S(R_S)}{\p r^2}S^2(\theta) - \frac{\p \sigma_1(R_S,\theta)}{\p r}S(\theta) \hspace{2em}&\text{on } \p B_{R_S}.
    \end{aligned}
    \right.
\end{equation}
Note that the general solution to the Poisson equation $-\Delta u + u =0$ in a disk is given by
$$u(r,\theta)= \sum_{n=0}^\infty A_n I_n(r)\cos(n\theta),$$
where the coefficients $A_n$ are determined by the boundary condition. We will next simplify the boundary condition as specified in \re{eqn:sigma2}. Recall that when $S(\theta) = \cos(n\theta)$,
\begin{equation}
    \label{sol:s1}
    \sigma_1(r,\theta)=\sigma_{1,n}(r)\cos(n\theta) = -\frac{I_1(R_S)I_n(r)}{I_0(R_S)I_n(R_S)}\cos(n\theta),
\end{equation}
hence
\begin{equation}
    \label{deri:s1}
    \frac{\p \sigma_1(R_S,\theta)}{\p r}= - \Big(\frac{I_1(R_S)I_{n+1}(R_S)}{I_0(R_S)I_n(R_S)} + \frac{n I_1(R_S)}{R_SI_0(R_S)} \Big)\cos(n\theta).
\end{equation}
By combining this with \eqref{ds2} and utilizing the formula of $M_n$ in \eqref{Mn}, we obtain
\begin{equation}\label{Mntilde}
    \begin{split}
        \sigma_2(R_S,\theta) &= -\frac12\Big(1-\frac{I_1(R_S)}{R_S I_0(R_S)}\Big)\cos^2(n\theta)+ \Big(\frac{I_1(R_S)I_{n+1}(R_S)}{I_0(R_S)I_n(R_S)} + \frac{n I_1(R_S)}{R_SI_0(R_S)} \Big)\cos^2(n\theta)\\
        &= \Big(\frac{2n+1}{2}\frac{I_1(R_S)}{R_S I_0(R_S)}-\frac12 + \frac{I_1(R_S)I_{n+1}(R_S)}{I_0(R_S)I_n(R_S)}\Big)\cos^2(n\theta)\\
        &= \Big(\frac12 + \frac{2n-3}2\frac{I_1(R_S)}{R_S I_0(R_S)} - M_n\Big) \cos^2(n\theta) \triangleq \widetilde{M}_n \cos^2(n\theta).
    \end{split}
\end{equation}
Using the double-angle formula
$$\cos^2(n\theta) = \frac{1+\cos(2n\theta)}{2},$$
we find that the boundary term $\sigma_2(R_S,\theta)$ is a linear combination of $1$ and $\cos(2n\theta)$. Therefore, the solution to \re{eqn:sigma2} should be a linear combination of $I_0(r)$ and $I_{2n}(r)\cos(2n\theta)$. With this insight, we solve \re{eqn:sigma2} as
\begin{equation}
    \label{sol:sigma2}
    \begin{split}
    \sigma_2(r,\theta) =\; \frac{\widetilde{M}_n}2 \frac{I_0(r)}{I_0(R_S)} + \frac{\widetilde{M}_n}2  \frac{I_{2n}(r)}{I_{2n}(R_S)}\cos(2n\theta).
    \end{split}
\end{equation}
To facilitate subsequent calculations, it is necessary to compute the term $\frac{\p \sigma_2(R_S,\theta)}{\p r}$. We proceed with this computation as follows:
\begin{equation}
    \begin{split}
        &\frac{\p \sigma_2(r,\theta)}{\p r} = \frac{\widetilde{M}_n}2 \frac{I_1(r)}{I_0(R_S)} +\frac{\widetilde{M}_n}2\Big(\frac{I_{2n+1}(r)}{I_{2n}(R_S)} + \frac{2n I_{2n}(r)}{r I_{2n}(R_S)}\Big)\cos(2n\theta),
    \end{split}
\end{equation}
hence
\begin{equation}\label{dersigma2}
    \begin{split}
        &\frac{\p \sigma_2(R_S,\theta)}{\p r} = \frac{\widetilde{M}_n}2\frac{I_1(R_S)}{I_0(R_S)} + \frac{\widetilde{M}_n}2 \Big(\frac{I_{2n+1}(R_S)}{I_{2n}(R_S)} + \frac{2n }{R_S }\Big)\cos(2n\theta).
    \end{split}
\end{equation}

In a similar manner, by substituting \re{hexpand2} and \re{hexpand3} into \re{eqn:pt} and \re{eqn:ps} and using Lemma \ref{kappalem}, we obtain the following systems for $\widetilde{p}_2$ and $p^*_2$:
\begin{equation}\label{eqn:p2t}
    \left\{
    \begin{aligned}
        &-\Delta \widetilde{p}_2 =0 \hspace{2em} &\text{in } B_{R_S},\\
        &\widetilde{p}_2(R_S,\theta) =\frac{2S S_{\theta\theta}+S^2 + \frac12 S_\theta^2}{R_S^3} -\frac12 \frac{\p^2 \widetilde{p}_S(R_S)}{\p r^2}  S^2 -\frac{\p \widetilde{p}_1(R_S,\theta)}{\p r}S
       \hspace{2em}&\text{on } \p B_{R_S},
    \end{aligned}
    \right.
\end{equation}
\begin{equation}\label{eqn:p2s}
    \left\{
    \begin{aligned}
        &-\Delta p^*_2 = \sigma_2 \hspace{2em} &\text{in } B_{R_S},\\
        &\widetilde{p}_2(R_S,\theta) = -\frac12\frac{\p^2 p^*_S(R_S)}{\p r^2}S^2 - \frac{\p p^*_1(R_S,\theta)}{\p r}S \hspace{2em}&\text{on } \p B_{R_S}.
    \end{aligned}
    \right.
\end{equation}
When $S(\theta)=\cos(n\theta)$, it follows from Section 3 that
\begin{gather}
    \label{sol:p1t} \widetilde{p}_1(r,\theta)= \widetilde{p}_{1,n}(r)\cos(n\theta)= \frac{(n^2-1)r^n}{R_S^{n+2}}\cos(n\theta),\\
    \label{sol:p1s} p^*_1(r,\theta) = p^*_{1,n}(r)\cos(n\theta)=\Big(-\frac{I_1(R_S)r^n}{R_S^n I_0(R_S)} + \frac{I_1(R_S)I_n(r)}{I_0(R_S)I_n(R_S)}\Big)\cos(n\theta),
\end{gather}
hence
\begin{gather}
    \label{deri:p1t} \frac{\p \widetilde{p}_1(R_S,\theta)}{\p r}= \frac{n(n^2-1)}{R_S^3}\cos(n\theta) ,\\
    \label{deri:p1s} \frac{\p p^*_1(R_S,\theta)}{\p r} = \frac{I_1(R_S)I_{n+1}(R_S)}{I_0(R_S)I_n(R_S)}\cos(n\theta).
\end{gather}

We first solve for $\widetilde{p}_2$. Substituting \re{dpt} and \re{deri:p1t} into the boundary condition in the system \re{eqn:p2t} and using the following equations,
\begin{gather*}
    S^2=\cos^2(n\theta)=\frac{1}{2}\big(1+\cos(2n\theta)\big),\\
    S S_{\theta\theta} = -n^2 \cos^2(n\theta) = -\frac{n^2}{2}\big(1+\cos(2n\theta)\big),\\
    S_\theta^2 =n^2\sin^2(n\theta) = \frac{n^2}{2}\big(1-\cos(2n\theta)\big),
\end{gather*}
we simplify the boundary condition to:
\begin{equation}
    \label{bdy:p2t}
    \begin{aligned}
        \widetilde{p}_2(R_S,\theta) =&\; \frac{1}{R_S^3}\Big[\Big(\frac12 - \frac{3n^2}4\Big) + \Big(\frac12-\frac{5n}4\Big)\cos(2n\theta)\Big] - \frac{n(n^2-1)}{R_S^3}\frac{1+\cos(2n\theta)}{2}\\
        =&\; \frac{1}{R_S^3}\Big(\frac12 - \frac{3n^2}4 - \frac{n(n^2-1)}{2}\Big)  + \frac{1}{R_S^3}\Big(\frac12-\frac{5n}4-\frac{n(n^2-1)}{2}\Big)\cos(2n\theta).
    \end{aligned}
\end{equation}
We observe that the boundary condition is a linear combination of $1$ and $\cos(2n\theta)$. Recalling the general solution to the Laplace equation $\Delta u = 0$ in a disk is
$$u(r,\theta)=\sum_{n=0}^\infty B_n r^n \cos(n\theta),$$
we note that the coefficients $B_n$ are determined by the boundary condition \re{bdy:p2t}. As a result, $\widetilde{p}_2$ is solved as
\begin{equation}
    \label{sol:p2t}
    \widetilde{p}_2(r,\theta)=\frac{1}{R_S^3}\Big(\frac12 - \frac{3n^2}4 - \frac{n(n^2-1)}{2}\Big) + \frac{r^{2n}}{R_S^{2n+3}}\Big(\frac12-\frac{5n}4-\frac{n(n^2-1)}{2}\Big)\cos(2n\theta).
\end{equation}
Since the first part of \re{sol:p2t} is a constant, it vanishes upon taking derivatives. Hence, we have
\begin{equation}
    \label{d:p2}
        \frac{\p \widetilde{p}_2(R_S,\theta)}{\p r} = \frac{2n}{R_S^4}\Big(\frac12-\frac{5n}4-\frac{n(n^2-1)}{2}\Big)\cos(2n\theta).
\end{equation}

Next, we proceed to solve for $p^*_2$ from the system \re{eqn:p2s}. We observe that
\begin{equation}\label{eqn:sp2}
    \left\{
    \begin{aligned}
        &\Delta(p^*_2+\sigma_2)=0  &\text{in } B_{R_S},\\
        &(p^*_2 + \sigma_2)(R_S,\theta) =  -\frac12\Big(\frac{\p^2 p^*_S(R_S)}{\p r^2}+\frac{\p^2 \sigma_S(R_S)}{\p r^2}\Big)S^2 - \Big(\frac{\p p^*_1(R_S,\theta)}{\p r}+ \frac{\p \sigma_1(R_S,\theta)}{\p r}\Big)S &\text{on } \p B_{R_S},
    \end{aligned}
    \right.
\end{equation}
where, by \re{ds2}, \re{dps2}, \re{deri:s1}, and \re{deri:p1s}, the boundary condition can be simplified as
\begin{equation*}
\begin{split}
    (p^*_2 + \sigma_2)(R_S,\theta) &= -\frac12 \frac{I_1(R_S)}{R_S I_0(R_S)}\frac{1+\cos(2n\theta)}{2} + \frac{n I_1(R_S)}{R_S I_0(R_S)}\frac{1+\cos(2n\theta)}{2}\\
    &= \frac{2n-1}{{\color{black}4}}\frac{I_1(R_S)}{R_S I_0(R_S)} + \frac{2n-1}{{\color{black}4}}\frac{I_1(R_S)}{R_S I_0(R_S)}\cos(2n\theta).
\end{split}
\end{equation*}
Reemploying the general solution of the Laplace equation, we have
\begin{equation}
    \label{sol:p2s}
    p^*_2(r,\theta) = \frac{2n-1}{{\color{black}4}}\frac{I_1(R_S)}{R_S I_0(R_S)} + \frac{r^{2n}}{R_S^{2n}}\frac{2n-1}{{\color{black}4}}\frac{I_1(R_S)}{R_S I_0(R_S)}\cos(2n\theta) - \sigma_2(r,\theta),
\end{equation}
hence
\begin{equation}
    \label{deri:p2s}
    \frac{\p p^*_2(R_S,\theta)}{\p r} = \frac{n(2n-1)}{ {\color{black}2}} \frac{I_1(R_S)}{R_S^2 I_0(R_S)}\cos(2n\theta) - \frac{\p \sigma_2(R_S,\theta)}{\p r},
\end{equation}
where $\frac{\p \sigma_2(R_S,\theta)}{\p r}$ was computed in \re{dersigma2}.

At this point, we shall rigorously establish the $\varepsilon^2$-order expansion for $\sigma$ and $p$ in the case $\p \Omega_\varepsilon: r=R_S+\varepsilon \cos(n\theta)$.

\begin{lem}\label{est2}
    Assume that $(\sigma,\widetilde{p},p^*)$ is the solution to System $(\mathcal{S})$ in the domain whose boundary is defined by $\p \Omega_\varepsilon: r=R_S + \varepsilon \cos(n\theta)$; $(\sigma_1,\widetilde{p}_1,p^*_1)$ is the solution to the system \re{eqn:sigma1} -- \re{eqn:p1s}, and the explicit solutions $\sigma_1$, $\widetilde{p}_1$, and $p^*_1$ are given in \re{sol:s1}, \re{sol:p1t}, and \re{sol:p1s}, respectively; and   $\sigma_2$, $\widetilde{p}_2$, and $p^*_2$ are solutions to the systems \re{eqn:sigma2}, \re{eqn:p2t}, and \re{eqn:p2s}, respectively, and are explicitly given by \re{sol:sigma2}, \re{sol:p2t}, and \re{sol:p2s}. Then,
    \begin{gather}
        \|\sigma-(\sigma_S + \varepsilon \sigma_1 + \varepsilon^2 \sigma_2)\|_{C^{3+\alpha}(\overline{\Omega}_\varepsilon)} \le C|\varepsilon|^3,\label{est:sigma2}\\
        \big\|\big(\widetilde{p}+\mu p^*\big)-\big(\widetilde{p}_S + \mu p^*_S + \varepsilon (\widetilde{p}_1+\mu p^*_1) + \varepsilon^2 (\widetilde{p}_2 + \mu p^*_2)\big)\big\|_{C^{1+\alpha}(\overline{\Omega}_\varepsilon)} \le C|\varepsilon|^3.\label{est:p2}
    \end{gather}
\end{lem}
\begin{proof}
    We first prove the estimate for $\sigma$. Note that $\sigma_S$, $\sigma_1$, and $\sigma_2$ are defined (by their explicit formulas) for all $r>0$, and they satisfy the same elliptic equations $-\Delta u + u = 0$ for all $r>0$. So, if we denote $\Psi_1=\sigma-(\sigma_S + \varepsilon \sigma_1 + \varepsilon^2 \sigma_2)$, then we have
    $$-\Delta \Psi_1 -\Psi_1 = 0 \hspace{2em}\text{in } \Omega_\varepsilon.$$
    We then check the boundary condition, on $\p\Omega_\varepsilon$,
    \begin{equation*}
        \begin{split}
            \Psi_1\Big|_{\p \Omega_\varepsilon} =& \Big(\sigma-(\sigma_S + \varepsilon \sigma_1 + \varepsilon^2 \sigma_2)\Big)\Big|_{r=R_S+\varepsilon \cos(n\theta)} \\
            =&1 - \sigma_S(R_S) - \frac{\p \sigma_S(R_S)}{\p r}\varepsilon\cos(n\theta) - \frac{\p^2 \sigma_S(R_S)}{\p r^2}\varepsilon^2 \cos^2(n\theta)-\varepsilon\sigma_1(R_S,\theta) \\
            &-\varepsilon\frac{\sigma_1(R_S,\theta)}{\p r}\varepsilon\cos(n\theta) -\varepsilon^2\sigma_2(R_S,\theta) + O(\varepsilon^3).
        \end{split}
    \end{equation*}
    Using \re{rss3}, \re{eqn:sigma1}, and \re{eqn:sigma2}, we find that all the lower-order terms cancel out, and $\Psi_1|_{\p \Omega_\varepsilon} = O(\varepsilon^3)$. Then the inequality \re{est:sigma2} follows by applying the Schauder estimates on $\Psi_1$.

    The proof to \re{est:p2} is similar. Let
    $$\Psi_2 = \big(\widetilde{p}+\mu p^*\big)-\big(\widetilde{p}_S + \mu p^*_S + \varepsilon (\widetilde{p}_1+\mu p^*_1) + \varepsilon^2 (\widetilde{p}_2 + \mu p^*_2)\big).$$
    Combining \re{eqn:pst}, \re{eqn:pss}, \re{eqn:pt}, \re{eqn:ps}, \re{eqn:p1t}, \re{eqn:p1s}, \re{eqn:p2t}, and \re{eqn:p2s}, we obtain
    $$-\Delta \Psi_2 = \mu\Psi_1\hspace{2em}\text{in } \Omega_\varepsilon,$$
    and on the boundary $\p \Omega_\varepsilon$, after canceling terms that arise from the Taylor series expansions,
    $$\Psi_2\Big|_{\p \Omega_\varepsilon} = \kappa\Big|_{r=R_S+\varepsilon \cos(n\theta)} - \frac{1}{R_S} - \varepsilon\frac{(n^2-1)\cos(n\theta)}{R_S^2}- \varepsilon^2\frac{(1-2n^2)\cos^2(n\theta)+\frac{n^2}2\sin^2(n\theta)}{R_S^3} + O(\varepsilon^3).$$
    By Lemma \ref{kappalem}, $\Psi_2|_{\p \Omega_\varepsilon} = O(\varepsilon^3)$ and the inequality \re{est:p2} follows immediately from \re{est:sigma2} and the Schauder theory.
\end{proof}

Based on Lemma \ref{est2}, we are now able to compute the second-order Fr\'echet derivative $\mathcal{F}_{\widetilde{R}\widetilde{R}}$ in \re{mu-d}. By \re{F1}, we have
\begin{equation}
    \mathcal{F}(\mu,\varepsilon\cos(n\theta)) = \Big(\frac{\p \widetilde{p}}{\p n} + \mu \frac{\p p^*}{\p n}\Big)\Big|_{r=R_S+\varepsilon \cos(n\theta)} = \Big(\nabla \widetilde{p} \cdot \vec{n} + \mu\nabla p^* \cdot \vec{n}\Big)\Big|_{r=R_S+\varepsilon \cos(n\theta)},
\end{equation}
where, by \re{hexpand2},
\begin{equation}\label{exp1}
    \begin{split}
        &\nabla \widetilde{p} \cdot \vec{n}\Big|_{r=R_S+\varepsilon \cos(n\theta)}\\
        =&\; \Big( \frac{\p \widetilde{p}}{\p r} \vec{e}_r +  \frac{\vec{e}_\theta}{R_S+\varepsilon\cos(n\theta)}\frac{\p \widetilde{p}}{\p \theta}\Big)\cdot \frac{1}{\sqrt{1+\big(\frac{-n\varepsilon\sin(n\theta)}{R_S+\varepsilon\cos(n\theta)}\big)^2}}\Big(\vec{e}_r - \frac{-n\varepsilon\sin(n\theta)\vec{e}_\theta}{R_S+\varepsilon\cos(n\theta)}\Big)\\
        =&\; \frac{\p \widetilde{p}_S(R_S)}{\p r} + \varepsilon\Big(\frac{\p^2 \widetilde{p}_S(R_S)}{\p r^2}\cos(n\theta) + \frac{\p \widetilde{p}_1(R_S,\theta)}{\p r}\Big) + \varepsilon^2\Big(\frac12\frac{\p^3 \widetilde{p}_S(R_S)}{\p r^3}\cos^2(n\theta)\\
        &+\frac{\p^2 \widetilde{p}_1(R_S,\theta)}{\p r^2}\cos(n\theta) + \frac{1}{R_S^2}\frac{\p \widetilde{p}_S(R_S,\theta)}{\p \theta}n\sin(n\theta)\Big) +\frac{\p \widetilde{p}_2(R_S,\theta)}{\p r}  + O(\varepsilon^3),
    \end{split}
\end{equation}
and, by \re{hexpand3},
\begin{equation}\label{exp2}
    \begin{split}
        &\nabla p^* \hspace{-3pt} \cdot \vec{n}\Big|_{r=R_S+\varepsilon \cos(n\theta)} \\
        =&\; \Big( \frac{\p p^*}{\p r} \vec{e}_r +  \frac{\vec{e}_\theta}{R_S+\varepsilon\cos(n\theta)}\frac{\p p^*}{\p \theta}\Big)\cdot \frac{1}{\sqrt{1+\big(\frac{-n\varepsilon\sin(n\theta)}{R_S+\varepsilon\cos(n\theta)}\big)^2}}\Big(\vec{e}_r - \frac{-n\varepsilon\sin(n\theta)\vec{e}_\theta}{R_S+\varepsilon\cos(n\theta)}\Big)\\
        =&\; \frac{\p p^*_S(R_S)}{\p r} + \varepsilon\Big(\frac{\p^2 p^*_S(R_S)}{\p r^2}\cos(n\theta) + \frac{\p p^*_1(R_S,\theta)}{\p r}\Big) + \varepsilon^2\Big(\frac12\frac{\p^3 p^*_S(R_S)}{\p r^3}\cos^2(n\theta)\\
        &+\frac{\p^2 p^*_1(R_S,\theta)}{\p r^2}\cos(n\theta) + \frac{1}{R_S^2}\frac{\p p^*_1(R_S,\theta)}{\p \theta}n\sin(n\theta) +\frac{\p p^*_2(R_S,\theta)}{\p r} \Big) + O(\varepsilon^3).
    \end{split}
\end{equation}
Hence,
\begin{equation}\label{F2-1}
\begin{split}
    &\mathcal{F}(\mu,\varepsilon \cos(n\theta))\\ =& \; \Big(\frac{\p \widetilde{p}_S(R_S)}{\p r} + \mu \frac{\p p^*_S(R_S)}{\p r}\Big) + \varepsilon\Big(\frac{\p^2 \widetilde{p}_S(R_S)}{\p r^2}\cos(n\theta) + \mu\frac{\p^2 p^*_S(R_S)}{\p r^2}\cos(n\theta)\\
    &+\frac{\p \widetilde{p}_1(R_S,\theta)}{\p r} + \mu \frac{\p p^*_1(R_S,\theta)}{\p r} \Big) + \varepsilon^2 \Big(\frac12\frac{\p^3 \widetilde{p}_S(R_S)}{\p r^3}\cos^2(n\theta) + \frac{\mu}2\frac{\p^3 p^*_S(R_S)}{\p r^3}\cos^2(n\theta)\\
    &+\frac{\p^2 \widetilde{p}_1(R_S,\theta)}{\p r^2}\cos(n\theta)+ \mu\frac{\p^2 p^*_1(R_S,\theta)}{\p r^2}\cos(n\theta) + \frac{1}{R_S^2}\frac{\p \widetilde{p}_1(R_S,\theta)}{\p \theta}n\sin(n\theta) \\
    &+\frac{\mu}{R_S^2}\frac{\p p^*_1(R_S,\theta)}{\p \theta}n\sin(n\theta) +\frac{\p \widetilde{p}_2(R_S,\theta)}{\p r} + \mu \frac{\p p^*_2(R_S,\theta)}{\p r}\Big) + O(\varepsilon^3).
\end{split}
\end{equation}
On the other hand, since $\mathcal{F}(\mu,0)=0$, Taylor series expansion gives
\begin{equation}
    \label{F2-2}
    \mathcal{F}(\mu,\varepsilon\cos(n\theta)) = \varepsilon \mathcal{F}_{\widetilde{R}}(\mu,0)[\cos(n\theta)] + \frac{\varepsilon^2 }2 \mathcal{F}_{\widetilde{R}\widetilde{R}}(\mu,0)[\cos(n\theta),\cos(n\theta)] + O(\varepsilon^3).
\end{equation}
Comparing the $\varepsilon^2$-order terms in \re{F2-1} and \re{F2-2}, we obtain
\begin{equation}
    \label{F2}
    \begin{split}
    &\mathcal{F}_{\widetilde{R}\widetilde{R}}(\mu,0)[\cos(n\theta),\cos(n\theta)]\\
    =&\; \frac{\p^3 \widetilde{p}_S(R_S)}{\p r^3}\cos^2(n\theta) +  \mu \frac{\p^3 p^*_S(R_S)}{\p r^3}\cos^2(n\theta) +2\frac{\p^2 \widetilde{p}_1(R_S,\theta)}{\p r^2}\cos(n\theta)\\
    &+ 2\mu\frac{\p^2 p^*_1(R_S,\theta)}{\p r^2}\cos(n\theta) + \frac{2}{R_S^2}\frac{\p \widetilde{p}_1(R_S,\theta)}{\p \theta}n\sin(n\theta) \\
    &+ \frac{2\mu}{R_S^2}\frac{\p p^*_1(R_S,\theta)}{\p \theta}n\sin(n\theta) + 2\frac{\p \widetilde{p}_2(R_S,\theta)}{\p r} + 2\mu \frac{\p p^*_2(R_S,\theta)}{\p r}.
    \end{split}
\end{equation}

In \re{mu-d}, $l\in Y^*$ satisfies $\text{Ker}\, l = \text{Im}\, \mathcal{F}_{\widetilde{R}}(\mu_n,0)$, where $\mu_n$ ($n\ge2$) are the bifurcation points derived in Section 3. By \re{Im} and Remark \ref{main2}, for every integer $n\ge2$,
$$\text{Im}\, \mathcal{F}_{\widetilde{R}}(\mu_n,0) = \text{Span}\{1,\cos(\theta),\cos(2\theta),\cdots,\cos((n-1)\theta),\cos((n+1)\theta),\cdots\}.$$
Therefore, we can take $l(s)=\int_0^{2\pi}\cos(n\theta)s\dif \theta$ which satisfies the requirement. Next, we proceed to compute the numerator and the denominator in formula \re{mu-d}.

We start by computing the denominator in \re{mu-d}. Using \re{F-mu-r} and recalling the sign of $M_n$ in \re{Mn-sign}, we have
\begin{equation}\label{denominator}
\begin{split}
\langle l,\mathcal{F}_{\mu\widetilde{R}}(\mu_n,0) [\cos(n\theta)]\rangle = \int_0^{2\pi}
    -M_n \cos^2(n\theta) \dif \theta = -M_n \pi < 0.
    \end{split}
\end{equation}

For the numerator in \re{mu-d}, we first use \re{sol:p1t} and \re{sol:p1s} to compute the partial derivatives with respect to $\theta$:
\begin{gather}
    \label{deri:pt1:t} \frac{\p \widetilde{p}_1(R_S,\theta)}{\p \theta} = \widetilde{p}_{1,n}(R_S)\big(-n\sin(n\theta)\big) = -\frac{n(n^2-1)}{R_S^2}\sin(n\theta),\\
    \label{deri:ps1:t} \frac{\p p^*_1(R_S,\theta)}{\p \theta} = p^*_{1,n}(R_S)\big(-n\sin(n\theta)\big) = \Big(-\frac{I_1(R_S)}{I_0(R_S)}+\frac{I_1(R_S)}{I_0(R_S)}\Big)\big(-n\sin(n\theta)\big)  = 0;
\end{gather}
and also the second-order $r$-derivative of $\widetilde{p}_1$ and $p^*_1$:
\begin{gather}
    \label{deri2:pt1} \frac{\p^2 \widetilde{p}_1(R_S,\theta)}{\p r^2} = \frac{\p^2 \widetilde{p}_{1,n}(R_S)}{\p r^2}\cos(n\theta) = \frac{n(n-1)(n^2-1)}{R_S^4}\cos(n\theta)\\
    \label{deri2:ps1} \frac{\p^2 p^*_1(R_S,\theta)}{\p r^2} = \frac{\p^2 p^*_{1,n}(R_S)}{\p r^2} \cos(n\theta) = \Big(\frac{I_1(R_S)}{I_0(R_S)} - \frac{I_1(R_S) I_{n+1}(R_S)}{R_S I_0(R_S)I_n(R_S)}\Big)\cos(n\theta).
\end{gather}
Now, substituting \re{dpt}, \re{dps3}, \re{d:p2}, \re{deri:p2s}, and \re{deri:pt1:t} -- \re{deri2:ps1} all into \re{F2} and using the double-angle formulas, we obtain
\begin{equation}\label{FRR}
    \begin{split}
        &\mathcal{F}_{\widetilde{R}\widetilde{R}}(\mu,0)[\cos(n\theta),\cos(n\theta)]\\
        =&\; \mu\Big(\frac{1}{R_S} - \frac{(2+R_S^2)I_1(R_S)}{R_S^2 I_0(R_S)}\Big)\frac{1+\cos(2n\theta)}{2} + \frac{2n(n-1)(n^2-1)}{R_S^4}\frac{1+\cos(2n\theta)}{2} \\
        &+ 2\mu \Big(\frac{I_1(R_S)}{I_0(R_S)} - \frac{I_1(R_S) I_{n+1}(R_S)}{R_S I_0(R_S)I_n(R_S)}\Big) \frac{1+\cos(2n\theta)}{2}  -\frac{2n^2(n^2-1)}{R_S^4}\frac{1-\cos(2n\theta)}{2} \\
        &+ \frac{2n-5n^2-2n^2(n^2-1)}{R_S^4}\cos(2n\theta) + \mu \frac{n(2n-1)I_1(R_S)}{R_S^2 I_0(R_S)}\cos(2n\theta) \\
        &- \mu \widetilde{M}_n \frac{I_1(R_S)}{I_0(R_S)}
        - \mu \widetilde{M}_n \Big(\frac{I_{2n+1}(R_S)}{I_{2n}(R_S)} + \frac{2n }{R_S }\Big)\cos(2n\theta)\\
        =&\; \big(\text{I} + \mu \text{II}\big) + \big(\text{III} + \mu \text{IV}\big)\cos(2n\theta),
    \end{split}
\end{equation}
where
\begin{equation*}
\begin{split}
    \text{I} =& -\frac{n(n^2-1)}{R_S^4}, \hspace{2em} \text{III} = -\frac{n(n^2-1)}{R_S^4} + \frac{2n-5n^2}{R_S^4},\\
    \text{II} =&\; \frac{1}{R_S}\Big(\frac12 - \frac{I_1(R_S)}{R_SI_0(R_S)}-\frac{I_1(R_S)I_{n+1}(R_S)}{I_0(R_S)I_n(R_S)}\Big) + \frac{I_1(R_S)}{I_0(R_S)}\Big(\frac12 - \widetilde{M}_n\Big)\\
    =&\; \frac{1}{R_S}\Big(M_n - \frac12 + \frac
    {I_1(R_S)}{R_S I_0(R_S)}\Big) + \frac{I_1(R_S)}{I_0(R_S)}\Big(\frac12 - \widetilde{M}_n\Big),
\end{split}
\end{equation*}
and
\begin{equation*}
\begin{aligned}
    \text{IV} =&\; \frac{1}{R_S}\Big(\frac12 - \frac{I_1(R_S)}{R_SI_0(R_S)}-\frac{I_1(R_S)I_{n+1}(R_S)}{I_0(R_S)I_n(R_S)}\Big) + \frac{I_1(R_S)}{I_0(R_S)}\Big(\frac12 + \frac{n(2n-1)}{R_S^2}\Big) \\
    &\;- \widetilde{M}_n \Big(\frac{I_{2n+1}(R_S)}{I_{2n}(R_S)} + \frac{2n }{R_S }\Big)\\
    =&\;\frac{1}{R_S}\Big(M_n - \frac12 + \frac
    {I_1(R_S)}{R_S I_0(R_S)}\Big) + \frac{I_1(R_S)}{I_0(R_S)}\Big(\frac12 + \frac{n(2n-1)}{R_S^2}\Big) \\
    &\;- \widetilde{M}_n \Big(\frac{I_{2n+1}(R_S)}{I_{2n}(R_S)} + \frac{2n }{R_S }\Big).
\end{aligned}
\end{equation*}
When $\mu=\mu_n$, it follows from \re{Mntilde}, \re{thm:mun}, and \re{RS} that
$$\mu_n M_n = \frac{n(n^2-1)}{R_S^3}$$
and
$$\mu_n \widetilde{M}_n = \frac{\mu_n}{2} + \mu_n\frac{2n-3}{2}\frac{I_1(R_S)}{R_SI_0(R_S)} - \mu_n M_n = \frac{\mu_n}{2} + \frac{(2n-3)\mu_n \widetilde{\sigma}}{4} -\frac{n(n^2-1)}{R_S^3}.$$
Hence, we obtain
\begin{equation*}
\begin{aligned}
    \text{I}+\mu_n \text{II} =\frac{\mu_n}{R_S}\Big(-\frac12+\frac{\widetilde{\sigma}}2\Big) + \mu_n\frac{I_1(R_S)}{I_0(R_S)}\Big(\frac12 - \widetilde{M}_n\Big),
\end{aligned}
\end{equation*}
and
\begin{equation*}
\begin{aligned}
    \text{III}+\mu_n \text{IV} =&\; \frac{2n-5n^2}{R_S^4}+ \frac{\mu_n}{R_S}\Big(-\frac12+\frac{\widetilde{\sigma}}2\Big) + \mu_n\frac{I_1(R_S)}{I_0(R_S)}\Big(\frac12 + \frac{n(2n-1)}{R_S^2}\Big) \\
    &\;- \mu_n \widetilde{M}_n \Big(\frac{I_{2n+1}(R_S)}{I_{2n}(R_S)} + \frac{2n }{R_S }\Big).
\end{aligned}
\end{equation*}
Note that we can express 1 and $\cos(2n\theta)$ in terms of $\cos(n\theta)$ and $\sin(n\theta)$. More specifically, Equation \re{FRR} can be rewritten as
\begin{equation*}
\begin{split}
        &\mathcal{F}_{\widetilde{R}\widetilde{R}}(\mu_n,0)[\cos(n\theta),\cos(n\theta)]\\
        =&\; \big(\text{I} + \mu_n \text{II}\big)\big(\cos^2(n\theta) + \sin^2(n\theta)\big) + \big(\text{III} + \mu_n \text{IV}\big)\big(\cos^2(n\theta) - \sin^2(n\theta)\big)\\
        =&\; \big(\text{I}+\mu_n \text{II} + \text{III}+\mu_n \text{IV}\big)\cos^2(n\theta) + \big(\text{I}+\mu_n \text{II} - \text{III}-\mu_n \text{IV}\big)\sin^2(n\theta)\\
        =&\;\Big[\frac{2n-5n^2}{R_S^4} + \frac{\mu_n}{R_S}\big(\widetilde{\sigma}-1\big) +  \mu_n \frac{I_1(R_S)}{I_0(R_S)}\Big(1-\widetilde{M}_n + \frac{n(2n-1)}{R_S^2}\Big)-\mu_n \widetilde{M}_n \Big(\frac{I_{2n+1}(R_S)}{I_{2n}(R_S)} \\
        &\;+ \frac{2n }{R_S }\Big) \Big]\cos^2(n\theta) + \Big[\frac{5n^2-2n}{R_S^4} + \mu_n\frac{I_1(R_S)}{I_0(R_S)}\Big(-\widetilde{M}_n - \frac{n(2n-1)}{R_S^2}\Big) + \mu_n \widetilde{M}_n \Big(\frac{I_{2n+1}(R_S)}{I_{2n}(R_S)}\\
        &\;+ \frac{2n }{R_S }\Big)\Big]\frac{1}{n^2}\Big(\frac{\dif \cos(n\theta)}{\dif \theta}\Big)^2.
        \end{split}
\end{equation*}
This demonstrates that $\mathcal{F}_{\widetilde{R}\widetilde{R}}(\mu_n,0)[\cos(n\theta), \cos(n\theta)]$ is a bilinear function, which facilitates its use in computing $\mathcal{F}_{\widetilde{R}\widetilde{R}}(\mu_n,0)[\cos(n\theta),\phi]$ in \re{eqn:mu-d2}. However, for computational simplicity, it is more practical to directly use Equation \eqref{FRR} when calculating the numerator in \eqref{mu-d}. We notice that
\begin{equation}\label{eqn:F2}
\begin{aligned}
     &\mathcal{F}_{\widetilde{R}\widetilde{R}}(\mu_n,0)[\cos(n\theta),\cos(n\theta)]\\
     =&\; \big(\text{I} + \mu_n \text{II}\big) + \big(\text{III} + \mu_n \text{IV}\big)\cos(2n\theta)\\
     =&\; \Big[\frac{\mu_n}{R_S}\Big(-\frac12+\frac{\widetilde{\sigma}}2\Big) + \mu_n\frac{I_1(R_S)}{I_0(R_S)}\Big(\frac12 - \widetilde{M}_n\Big)\Big] + \Big[\frac{2n-5n^2}{R_S^4}+ \frac{\mu_n}{R_S}\Big(-\frac12+\frac{\widetilde{\sigma}}2\Big)\\
     &\;+ \mu_n\frac{I_1(R_S)}{I_0(R_S)}\Big(\frac12 + \frac{n(2n-1)}{R_S^2}\Big) - \mu_n \widetilde{M}_n \Big(\frac{I_{2n+1}(R_S)}{I_{2n}(R_S)} + \frac{2n }{R_S }\Big) \Big]\cos(2n\theta)\\
    =&\;\Lambda_1 + \Lambda_2 \cos(2n\theta)
    \end{aligned}
\end{equation}
is a linear combination of 1 and $\cos(2n\theta)$. Since $\cos(n\theta)$ is orthogonal to both 1 and $\cos(2n\theta)$, it yields
\begin{equation}
    \label{num}\langle l,\mathcal{F}_{\widetilde{R}\widetilde{R}}(\mu_n,0)[\cos(n\theta),\cos(n\theta)]\rangle = \Lambda_1\int_0^{2\pi} \cos(n\theta)\dif \theta + \Lambda_2\int_0^{2\pi} \cos(n\theta)\cos(2n\theta)\dif \theta = 0.
\end{equation}

Therefore, by combining \re{denominator} and \re{num} and applying Theorem \ref{shi}, we find that $\mu_n'(0)=0$ for all $n\ge 2$. Here, $(\mu_n(\varepsilon), \widetilde{R}n(\varepsilon))$ represent the bifurcation solutions obtained in Theorem \ref{main1}. We summarize our findings in the following theorem:
\begin{thm}
Let $\Gamma_{2,n}=\{(\mu_n(\varepsilon),\sigma_n(\varepsilon,\cdot),\widetilde{p}_n(\varepsilon,\cdot),p^*_n(\varepsilon,\cdot),\partial \Omega_{\varepsilon,n}): |\varepsilon|\ll 1\}$ be the solution branches for System $(\mathcal{A})$ obtained in Theorem \ref{main1}, where $n\ge 2$ is an integer. Then, $\mu_n'(0)=0$, and the bifurcation at $(\mu_n,0)$ is a pitchfork bifurcation.
\end{thm}

\section{Discussion}
The bifurcation theory has been extensively developed for ODEs and PDEs within fixed domains. However, its application to free boundary problems remains largely unexplored. Although most studies on free boundary tumor growth models have established the existence of symmetry-breaking bifurcations, they have not investigated the structure of these bifurcations. One of the primary challenges is the nonlinear decoupling of the free boundary, which makes the calculation of $\mu'(0)$ in \re{mu-d} complicated.

In this paper, we first applied the Crandall-Rabinowitz Theorem to demonstrate the existence of bifurcations. It is important to note that, in the first requirement of the Crandall-Rabinowitz Theorem, the special solution should be independent of the bifurcation parameter. In our model, the radially symmetric stationary solution $p_S$ depends on the bifurcation parameter $\mu$. To address this dependency, we introduced a decomposition of $p$ as $p=\widetilde{p}+\mu p^*$ and refined the proof methodologies found in \cite{friedman2008stability,angio1,wang2014bifurcation,zhao2020symmetry}. Moreover, we expanded the solution $(\sigma,\widetilde{p},p^*)$ to $O(\varepsilon^2)$ order and showed that $\mu'(0)=0$ at each bifurcation point. It indicates that all symmetry-breaking bifurcations established by the Crandall-Rabinowitz Theorem are pitchfork bifurcations.

We employed the numerical method described in \cite{hao2011three} to compute bifurcation solutions within System $(\mathcal{A})$. The discretization was set up with $N_R=20$ and $N_\theta=32$, and we focused on computing bifurcation solutions near the bifurcation point $\mu =\mu_2 \approx 8.6445$. For small values of $\varepsilon$, we observed two distinct bifurcation solutions: one corresponding to a positive $\varepsilon$ and the other to a negative $\varepsilon$. Figure \ref{fig:bifurcation} exhibits contour plots of the boundaries for these two bifurcation solutions. Pitchfork bifurcations typically occur in systems with symmetry and exhibit symmetry in the bifurcation diagram. The presence of two bifurcation solutions around the same bifurcation point suggests a pitchfork bifurcation, which is consistent with our theoretical findings.

\begin{figure}[H]
    \centering
    \includegraphics[width=0.8\linewidth]{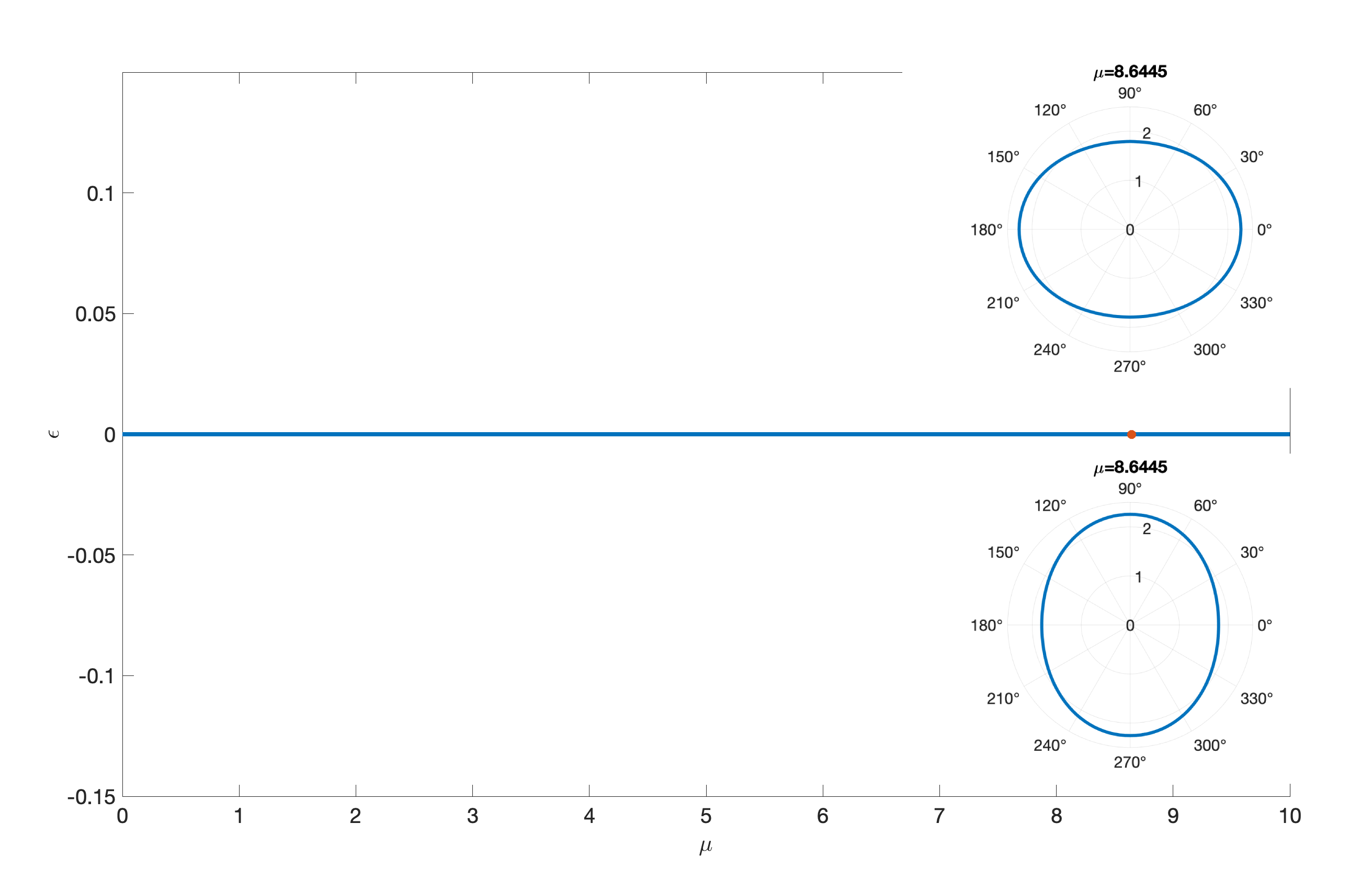}
    \caption{Contour plots of the boundaries for the two bifurcating solutions near $\mu=8.6445$, where $R_S=2$.}
    \label{fig:bifurcation}
\end{figure}

To further determine the direction of the pitchfork bifurcations, it is necessary to proceed to the next order, $O(\varepsilon^3)$, and determine the sign of $\mu''(0)$ using \eqref{eqn:mu-d2}. This task involves more complex and intricate calculations, so we leave it for our future investigation.

  \bibliographystyle{elsarticle-num-names}

\end{document}